\documentclass[12pt]{article}
\usepackage{latexsym, amsmath, amsthm, amsfonts} 
\usepackage{bbm}
\usepackage{enumerate, amssymb, xspace}
\usepackage[mathscr]{eucal}
\usepackage{graphicx} 
\usepackage{rotating}
\usepackage{hyperref}
\usepackage{color,colordvi}

\usepackage{mathtools}   %% for \xmapsto

\usepackage{tikz-cd}
\usetikzlibrary{matrix,arrows}
\usepackage{braids}

  \newcommand\Cite[2] {\cite[#1]{#2}}

\def\piclabetsep {0.45pt}
\def\piclinesize {1.32pt}

\def\xxcolor     {green!55!black}
\def\Lcolor      {red!15!black}
\def\locco       {blue}

\newcommand\drawCloseToTwist[6] { 
                 \draw[very thick,color=#6] (#1+#4,#2) -- (#1+#4,#2-#5) ;
                 \drawLeftEval {#1+0.5*#4} {#2+0.5*#4} {0.5*#4} {0} {#6} ;
                 \drawRightCoeval {#1+0.5*#4} {#2-#5-0.5*#4} {0.5*#4} {0} {#6} ;
	}
\newcommand\drawLeftCoeval[5] {
                 \draw[very thick,color=#5,->] (#1-#3,#2+#3) arc(180:277:#3 cm);
                 \draw[very thick,color=#5] (#1,#2) arc(270:360:#3 cm);
                 \draw[very thick,color=#5] (#1-#3,#2+#3) -- (#1-#3,#2+#3+#4);
                 \draw[very thick,color=#5] (#1+#3,#2+#3) -- (#1+#3,#2+#3+#4); }
\newcommand\drawLeftEval[5] {
                 \draw[very thick,color=#5,->] (#1-#3,#2-#3) arc(180:83:#3 cm);
                 \draw[very thick,color=#5] (#1,#2) arc(90:0:#3 cm);
                 \draw[very thick,color=#5] (#1-#3,#2-#3) -- (#1-#3,#2-#3-#4);
                 \draw[very thick,color=#5] (#1+#3,#2-#3) -- (#1+#3,#2-#3-#4); }
\newcommand\drawRightEval[5] {
                 \draw[very thick,color=#5,->] (#1+#3,#2-#3) arc(0:97:#3 cm);
                 \draw[very thick,color=#5] (#1,#2) arc(90:180:#3 cm);
                 \draw[very thick,color=#5] (#1-#3,#2-#3) -- (#1-#3,#2-#3-#4);
                 \draw[very thick,color=#5] (#1+#3,#2-#3) -- (#1+#3,#2-#3-#4); }
\newcommand\drawRightCoeval[5] {
                 \draw[very thick,color=#5,->] (#1+#3,#2+#3) arc(360:263:#3 cm);
                 \draw[very thick,color=#5] (#1,#2) arc(270:180:#3 cm);
                 \draw[very thick,color=#5] (#1-#3,#2-#3) -- (#1-#3,#2-#3-#4);
                 \draw[very thick,color=#5] (#1+#3,#2-#3) -- (#1+#3,#2-#3-#4); }

  \newcommand\drawDinat[5] {\filldraw[fill=green!40, draw=\Lcolor, thick]
                           (#1-#3,#2) -- (#1+#3,#2) -- (#1,#2-#4) -- cycle;
                           \node [red,right=-3pt] at  (#1+#3,#2-#4/2)  {$ #5 $}; }

  \newcommand\drawDinatTleft[5] {\filldraw[fill=green!40, draw=\Lcolor, thick]
                           (#1-#3,#2) -- (#1+#3,#2) -- (#1,#2-#4) -- cycle;
                           \node [red,left=-3pt] at  (#1+#3,#2-#4/2)  {$ #5 $}; }

  \newcommand\drawDinatWLegs[6] {\drawDinat{#1}{#2}{#3}{#4}{#5}
	           \draw[very thick,color=\xxcolor] (#1-#3/2,#2-#4/2) -- (#1-#3/2,#2-#4-#6);
		   \draw[very thick,color=\xxcolor] (#1+#3/2,#2-#4/2) -- (#1+#3/2,#2-#4-#6);}
  \newcommand\drawDinatStdWLegs[4] {\drawDinatWLegs {#1}{#2} {.2} {.18} {#3} {#4}}

  \newcommand\drawDinatTleftWLegs[6] {\drawDinatTleft{#1}{#2}{#3}{#4}{#5}
	           \draw[very thick,color=\xxcolor] (#1-#3/2,#2-#4/2) -- (#1-#3/2,#2-#4-#6);
		   \draw[very thick,color=\xxcolor] (#1+#3/2,#2-#4/2) -- (#1+#3/2,#2-#4-#6);}
  \newcommand\drawDinatStdTleftWLegs[4] {\drawDinatTleftWLegs {#1}{#2} {.2} {.18} {#3} {#4}}

  \newcommand\drawDinatWLegsLocco[6] {\drawDinat{#1}{#2}{#3}{#4}{#5}
	           \draw[very thick,color=\locco] (#1-#3/2,#2-#4/2) -- (#1-#3/2,#2-#4-#6);
		   \draw[very thick,color=\locco] (#1+#3/2,#2-#4/2) -- (#1+#3/2,#2-#4-#6);}
  \newcommand\drawomegaF[2] {\node[fill=blue!15,draw=\Lcolor,thick,double]
                     at (#1,#2) {\hspace*{14pt}$\kappa^{\phantom-}_{\phantom:}$\hspace*{7pt}}; }
  \newcommand\drawomegaFm[2] {\node[fill=blue!30,draw=\Lcolor,thick,double]
                     at (#1,#2) {\hspace*{12pt}$\omegaFm_{\phantom:}$\hspace*{9pt}}; }
  \newcommand\drawomegaL[2] {\node[fill=red!15,draw=\Lcolor,thick,double,rounded corners]
                     at (#1,#2) {\hspace*{14pt}$\omega^{\phantom-}_{\phantom:}$\hspace*{7pt}}; }
  \newcommand\drawomegaLm[2] {\node[fill=red!30,draw=\Lcolor,thick,double,rounded corners]
                     at (#1,#2) {\hspace*{13pt}$\omegaLm_{\phantom:}$\hspace*{8pt}}; }

\def\A             {{\ensuremath{\mathcal A}}}

\def\apo           {{\mathrm s}}
\def\apoL          {\apo}
\def\be            {\begin{equation}}
\def\bearl         {\begin{array}{l}}
\def\bearll        {\begin{array}{ll}}

\def\BlA           {\mathrm{Bl}^{}_\text{ann}}

\def\Blin          {\mathrm{Bl}^\text{in}_\text{b.s.}}
\def\Blinout       {\mathrm{Bl}^\text{in/out}_\text{b.s.}}

\def\Blout         {\mathrm{Bl}^\text{out}_\text{b.s.}}
\def\boti          {\,{\boxtimes}\,}
\def\C             {{\ensuremath{\mathcal C}}}

\def\Cb            {{\ensuremath{\overline{\mathcal C}}}}
\def\CbC           {{\ensuremath{\overline{\mathcal C}\boti\mathcal C}}}
\def\cb            {\beta}
\def\cB            {\gamma}
\def\cc            {(\cb^{\C})}
\def\ccc           {(\cb^{\Cb\boxtimes\C})}
\def\Ccc           {\cb^{\Cb\boxtimes\C}}

\def\chii          {{\raisebox{.15em}{$\chi$}}}
\def\chiv          {\chii^{\V}}
\def\cir           {\,{\circ}\,}

\def\cochi         {{\raisebox{.15em}{$\widehat\chi$}}}
\def\coev          {{\mathrm{coev}\!}}
\def\complex       {{\ensuremath{\mathbbm C}}}

\def\cvp           {\,{*}\,}
\def\cvpF          {\,{*_{\!F}}\,}
\def\cvpL          {\,{*_{\!L}}\,}
\def\cz            {\cb^{\Z}}

\def\cza           {(\cb^{\Z(\A)})}
\def\D             {{\ensuremath{\mathcal D}}}

\def\DeltaF        {\Delta_F}
\def\DeltaL        {\Delta}

\def\dim           {\mathrm{dim}}

\def\ee            {\end{equation}}

\def\eear          {\end{array}}
\def\eeta          {\mathrm e}

\def\EndC          {{\ensuremath{\mathrm{End}_\C}}}
\def\EndId         {\mathcal E\!nd_\koc(\Id_\C)} 
\def\EndID         {\mathcal E\!nd(\Id_\C)}
\def\EndIDA        {\mathcal E\!nd(\Id_{A\Mod})}
\def\EndZ          {{\ensuremath{\mathrm{End}_{\Z(\C)}}}}

\def\eps           {\varepsilon}
\def\epsF          {\eps_{\!F}}
\def\epsL          {\eps}
\def\erf           {\eqref }
\def\eq            {\,{=}\,}
\def\etaF          {\eta}
\def\etaL          {\eta}
\def\ev            {{\mathrm{ev}\!}}
\def\Fc            {{\mathring F}}
\def\findim        {fini\-te-di\-men\-si\-o\-nal}
\newcommand\fopi[1]{\footnote{TMP: See picture #1 in the handwritten notes {\tt jf\_170911\_fgsSpics.pdf}.}}
\newcommand\Fopi[2]{\footnote{TMP: See picture #1#2 in the handwritten notes {\tt jf\_170911\_fgsSpics.pdf}.}}
\renewcommand\fopi[1]{}
\renewcommand\Fopi[2]{}
\newcommand\Fopj[2]{}
\def\GC            {{\ensuremath{\Xi_\C}}}  

\def\Hom           {{\ensuremath{\mathrm{Hom}}}}
\def\HomC          {{\ensuremath{\mathrm{Hom}_\C}}}
\def\HomZ          {{\ensuremath{\mathrm{Hom}_{\Z(\C)}}}}
\newcommand\hsp[1] {\mbox{\hspace{#1 em}}}
\def\id            {\mbox{\sl id}}
\def\Id            {\mbox{\sl Id}}

\def\ii            {{\rm i}}
\def\iF            {\iZe} 
\def\iL            {\iZe}
\def\iN            {\,{\in}\,}
\def\Itemize       {\def\leftmargini{1.04em}~\\[-1.66em]\begin{itemize}\addtolength\itemsep{-7pt}}
\newcommand\iZ[1]  {\imath^{Z{\sss(}#1{\sss)}}}
\def\iZe           {\iZ\one}
\newcommand\iZZ[1] {\imath^{Z\circ Z{\sss(}#1{\sss)}}}
\def\jL            {\jZe}
\def\jZe           {\jmath^{Z{\sss(}\one{\sss)}}}

\def\ko            {\Bbbk} 
\def\koc           {\complex} 
\def\la            {{\rm l.a.}}
\def\lambdaL       {\lambda}
\def\LambdaL       {\Lambda}

\def\M             {{\ensuremath{\mathcal M}}}
\def\mm            {\mathrm m}
\def\Mod           {\text{-mod}}

\def\mul           {\mu}
\def\mulF          {\mul}
\def\mulL          {\mul}

\newcommand\nxl[1] {\\[#1mm]}
\newcommand\Nxl[1] {\\[-1.3em]\\[#1mm]}

\def\Ol            {}
\def\Om            {\Omega}
\def\Omm           {\Om^{-1}}
\def\OmT           {\widetilde{\Om}}
\def\omegaF        {\kappa^{}} 
\def\omegaFm       {\kappa^-} 
\def\omegaL        {\omega^{}}
\def\omegaLm       {\omega^-} 
\def\one           {{\bf1}}

\def\oti           {\,{\otimes}\,}

\def\oticc         {\,{\otimes_{\Cb\otimes\C}}\,}
\def\oticx         {\,{\otimes_\ko}\,}
\def\otiz          {\,{\otimes_{\Z(\C)}}\,}
\def\qquand        {\qquad{\rm and}\qquad}
\def\ra            {{\rm r.a.}}
\def\rrangle       {\rangle\!\rangle}
\def\rep           {representation}

\def\scs           {\scriptstyle}

\def\scs           {\scriptstyle}
\def\soo           {S^{\displaystyle\circ\!\!\circ}}
\def\sooi          {\widetilde S^{\displaystyle\circ\!\!\circ}}
\def\sse           {\scriptsize }
\def\ssg           {\scriptstyle}
\def\sss           {\scriptscriptstyle}
\def\T             {{\ensuremath{\mathbb T}}}
\def\tcoev         {\widetilde{\mathrm{coev}\!}}
\def\tild          {\hat }
\def\tev           {\widetilde{\mathrm{ev}\!}}

\def\To            {\,{\to}\,}
\def\TO            {\,{\Rightarrow}\,}
\def\tphi          {\varphi}
\def\tpsi          {\psi}

\def\vphi          {\tild\varphi}

\def\vpsi          {\tild\psi}
\def\V             {\ensuremath{\mathfrak V}}

 \newcommand\void[1]{}
\def\W             {{\ensuremath{\mathcal W}}}
\newcommand\xarr[1]{\xrightarrow{~#1\,}}
\def\xcong         {\,{\xrightarrow{~\cong\,}}\,}
\def\Z             {{\ensuremath{\mathcal Z}}}
\def\ZA            {{\ensuremath{\mathcal Z(\mathcal A)}}}
\def\Ze            {Z(\one)}
\def\ZC            {{\ensuremath{\mathcal Z(\mathcal C)}}}
\def\zet           {{\ensuremath{\mathbb Z}}}
\def\ZZ            {Z}

\newtheorem{thm}{Theorem}
\newtheorem{lem}[thm]{Lemma}

\newtheorem{pstl}[thm]{Postulate}

\theoremstyle{definition}
\newtheorem{Example}[thm]{Example}

\newtheorem{Definition}[thm]{Definition}
\newtheorem{rem}[thm]{Remark}

\begin{document}

 \numberwithin{equation}{section}
 \numberwithin{thm}{section}

\begin{flushright}
   {\sf ZMP-HH/17-29}\\ 
   {\sf Hamburger$\;$Beitr\"age$\;$zur$\;$Mathematik$\;$Nr.$\;$709} \\[2mm] ~ 
\end{flushright}
\vskip 3.5em
                                   
\begin{center}
\begin{tabular}c \Large\bf The logarithmic Cardy case: Boundary states and annuli
\end{tabular}

\vskip 2.6em
                                     
  J\"urgen Fuchs\,$^{\,a}$,
  ~~Terry Gannon\,$^{\,b}$, 
  ~~Gregor Schaumann\,$^{\,c}$ 
  ~~and~~~ Christoph Schweigert\,$^{\,d}$

\vskip 9mm

  \it$^a$
  Teoretisk fysik, \ Karlstads Universitet\\
  Universitetsgatan 21, \ S\,--\,651\,88\, Karlstad, Sweden \\[7pt]
  \it$^b$
  Department of Mathematical and Statistical Sciences, University of Alberta, \\
  Edmonton, Alberta T6G 2G1, Canada \\[7pt]
  \it$^c$
  Fakult\"at f\"ur Mathematik, Universit\"at Wien, Austria\\[7pt]
  \it$^d$
  Fachbereich Mathematik, \ Universit\"at Hamburg\\
  Bereich Algebra und Zahlentheorie\\
  Bundesstra\ss e 55, \ D\,--\,20\,146\, Hamburg

\end{center}
                     
\vskip 1.3em

\noindent{\sc Abstract}\\[3pt]
We present a model-independent study of boundary states in the Cardy case that
covers all conformal field theories for which the representation category
of the chiral algebra is a -- not necessarily semisimple -- modular tensor
category. This class, which we call finite CFTs, includes all rational theories,
but goes much beyond these, and in particular comprises many logarithmic conformal
field theories.
\\
We show that the following two postulates for a Cardy case are compatible beyond 
rational CFT and lead to a universal description of boundary states that realizes
a standard mathematical setup: First, for bulk fields, the pairing of left and
right movers is given by (a coend involving) charge conjugation; and second, the
boundary conditions are given by the objects of the category of chiral data. 
For rational theories our proposal reproduces the familiar result for the 
boundary states of the Cardy case. Further, with the help of sewing we 
compute annulus amplitudes. Our results show in particular that these possess an
interpretation as partition functions, a constraint that for generic finite CFTs
is much more restrictive than for rational ones.

\newpage
\tableofcontents
\newpage
%%%%%%%%%%%%%%%%%%%%%%%%%%%%%%%%%%%%%%%%%%%%%%%%%%%%%%%%%%%%%%%%%%%%%%%%

\section{Introduction}\label{sec:intro}

Two-dimensional conformal field theory, or CFT for short, is of fundamental importance 
in many areas, including the theory of two-dimensional critical systems in statistical 
mechanics, string theory, and quasi one-dimensional condensed matter systems. For
understanding issues like percolation probabilities, open string perturbation theory 
in D-brane backgrounds, or defects in condensed matter physics, one must
study CFT on surfaces with boundary. Of particular interest in applications is one of
the simplest surfaces of this type, namely a disk with one bulk field insertion. From 
a more theoretical perspective, these correlators offer the most direct way
to gain insight into boundary conditions, and have been frequently used to this end.

Basic symmetries of a conformal field theory are encoded in a chiral symmetry algebra,
which can be realized as a vertex operator algebra. Here we consider the situation that
the representation category \C\ of the chiral algebra has the structure of a ribbon 
category; this structure encodes in particular information about conformal weights and 
about the braiding and fusing matrices in a basis independent form. We will be interested 
in theories for which the category \C\ exhibits suitable finiteness properties and has
dualities and a non-degenerate braiding (for details see Definition \ref{def:modular}).
We refer to ribbon categories with the relevant properties as \emph{modular tensor
categories}. Modular categories in this sense are \emph{not} required to
be semisimple, and indeed there are many interesting systems, such as critical 
dense polymers \cite{dupl2}, for which \C\ is non-semisimple. For brevity, we 
will refer to conformal field theories whose chiral data are described by such a 
category as \emph{finite conformal field theories}. The class of finite CFTs 
includes, besides all rational CFTs, in particular all rigid finite logarithmic
CFTs. In terms of vertex operator algebras, the relevant notion of finiteness is,
basically, $C_2$-\emph{cofiniteness} \cite{miya8,huan28}; see \cite{crga3}
for precise statements and examples.

\medskip

In the present paper we are concerned with specific correlators for finite CFTs: with 
boundary states and with annulus partition functions. Boundary states and boundary
conditions are a feature of \emph{full} local conformal field theory, in which left-
and right-movers are adequately combined. In the special case of \emph{rational} CFTs,
for which \C\ is a semisimple modular tensor category, the structure of a full 
conformal field theory is fully understood \cite{fuRs4,fjfrs2} and can be implemented
in the framework of vertex operator algebras \cite{huKo3}. This includes in particular
the proper description as well as classification of boundary conditions. The simplest
possibility -- known as the Cardy case -- is that the boundary conditions are just
the objects of the tensor category \C, while in the general case they are the objects 
of a module category over \C.

Beyond semisimplicity, much less is 
known, but there has been substantial recent progress. Specifically, structural
properties of the space of bulk fields and their role for fulfilling the modular
invariance and sewing constraints have been understood \cite{fuSc22}, and systematic
model-independent results for correlators of finite CFTs on closed world sheets
have been obtained \cite{fuSs4,fuSs5,fuSc22}. 
In contrast, no model-independent results are available for correlators of 
non-semisimple finite CFTs on world sheets with boundary. 

\medskip

The present paper takes the first steps towards filling this gap. Concerning boundary
conditions and boundary states, our starting point consists of the following two
statements which can be expected to be valid under very general circumstances,
even beyond the realm of finite CFTs:
\\[3pt]
{\bf(BC)}~\,
First, the \emph{boundary conditions} for a given local conformal field theory should
be the objects of some category \M. This category may be realized in various guises, 
e.g.\ as the (homotopy) category of matrix factorizations in a Landau-Ginzburg formulation,
or as a category of modules over a Frobenius algebra in the TFT approach \cite{fuRs4}
to rational CFT. For a finite CFT based on a modular tensor category \C, the
category \C\ itself is a natural candidate for the category of boundary conditions.
If this is a valid choice and thus determines a consistent local CFT, then it is appropriate 
to refer to that full local CFT, following the parlance for rational theories, as the 
\emph{Cardy case}.
\\[3pt]
{\bf(BS)}~\,
Second, an essential feature of a \emph{boundary state} is that it associates to a given
boundary condition an element of some vector space. In a Landau-Ginzburg formulation, 
this space is a center (or its derived version, a Hochschild complex). 
In a more abstract approach to conformal field theory, the appropriate notion is the
center of the category \C, i.e.\ the space $\EndID$ of natural endo-trans\-formations
of the identity functor of \C. (This generalizes 
the fact that for the category $A$\Mod\ of modules over an associative algebra $A$, 
$\EndIDA$ can be identified with the center of $A$ as an algebra.) With the help of
standard categorical manipulations this vector space can be expressed as
  \be
  \EndID \,= \int_{c\in \C} \HomC(c,c)
  \,\cong \int_{c\in \C} \HomC(c^\vee {\otimes}\, c,\one) \,\cong\, \HomC(L,\one) 
  \label{eq:EndId=Hom}
  \ee
with $\one$ the tensor unit of \C\ and with the object $L$ of \C\ given by
$L \eq \int^{c\in\C}c^\vee {\otimes}\, c$. (The \emph{end} $\int_c$ and \emph{coend}
$\int^c$ appearing here are categorical limit and colimit constructions, respectively;
for the functors in question they exist in any finite tensor category.)

\medskip

Now the map from boundary conditions to the vector space $\HomC(L,\one)$ is a 
decategorification. It is thus natural to expect that it factorizes over the Grothendieck 
ring $K_0(\C)$, which is the decategorification of the category \C. Such a factorization
over the Grothendieck ring is generally afforded by \emph{characters}. We should therefore
expect that boundary states are characters for representations of some suitable algebraic
structure; as we will see, the latter
is precisely the object $L$, endowed with a natural Hopf algebra structure. Let us note
that a similar description is known from two-dimensional topological field theories, as
studied in \Cite{Sect.\,7}{caWil}. In that case the role of the center is played by the
zeroth Hochschild homology of smooth projective schemes (or of more general spaces), 
and the homomorphism from the
Grothendieck ring to the center is given by the Chern character \Cite{Prop.\,13}{caWil}.

\subsection{Boundary states in rational CFT}\label{sec:1.1}

As we will now explain, the paradigm outlined above is indeed realized in the semisimple
case. In that case, boundary states can be regarded as the characters of specific $L$-modules
which are given by the objects of \C\ together with a canonical $L$-action on them.
We will now give a detailed account of this interpretation of the structure of boundary
states of a rational CFT with semisimple modular tensor category \C\ in the Cardy case.
In the Cardy case of a rational CFT, one first selects a finite set $(x_i)_{i\in I}$ of 
representatives for the isomorphism classes of simple objects of \C. 
Boundary states are then conventionally written as linear combinations of so-called
\emph{Ishibashi states}; for each $i\iN I$ there is one Ishibashi state $|i\rrangle$.
An Ishibashi state is in fact nothing but a canonical vector spanning a space of 
two-point conformal blocks on the sphere, namely the one with the two chiral
insertions given by $x_i$ and $x_{\overline i}$.
Boundary conditions are thus labeled by objects $x$ of \C, and elementary boundary 
conditions by isomorphism classes of
simple objects $x_a$ of \C\ with $a \iN I$. The boundary state $|x_a\rangle$ 
associated with the elementary boundary condition $x_a$ is expanded in Ishibashi states as
  \be
  |x_a\rangle = \sum_{i\in I} \frac{S_{ia}}{\sqrt{S_{i0}}} \, |i\rrangle \,,
  \label{|x_a>}
  \ee
where $S_{ij}$ are the entries of the modular S-matrix --
the non-degenerate matrix which represents the transformation $\tau\,{\mapsto}\,{-}1/\tau$
on the (vertex algebra) characters of the theory --
and $0 \iN I$ is the label for the identity field, i.e.\ for the tensor unit $\one$ of \C.

The formula \eqref{|x_a>} can be conveniently understood via the relation \cite{fffs3}
to three-dimen\-si\-o\-nal topological field theory. Namely, the boundary state
$|x_a\rangle$ can be constructed as the topological invariant that the TFT functor
$\mathrm{tft}$ associates to a certain ribbon link in the three-ball:
     \def\locpa {1.4} 
  \be
  \raisebox{3.3em} {$ |x_a\rangle ~=~ \displaystyle\sum_{i\in I} ~ \mathrm{tft} \Big( ~~~ $}
  \begin{tikzpicture}
  \shadedraw[color=gray] (0,\locpa) circle (\locpa);
  \draw[very thick,color=red] (-0.08*\locpa,1.35*\locpa) arc (95:445:0.85*\locpa cm and 0.35*\locpa cm);
  \draw[very thick,color=blue] (0,0) -- (0,0.58*\locpa);
  \draw[very thick,color=blue] (0,0.74*\locpa) -- (0,2*\locpa);
  \node (a) at (0.48*\locpa,0.57*\locpa) {$\scs x_a $};
  \node (i) at (0.23,0.24) {$\scs x_i $};
  \end{tikzpicture}
  \raisebox{3.3em} {$ ~~~ \Big)\,. $}
  \label{eq:|x_a>}
  \ee
By construction this is a vector in the space of two-point conformal blocks on the sphere;
expanding it in the basis 
     \def\locpa {1.4} 
  \be
  \raisebox{3.3em} {$ |i\rrangle ~=~ \mathrm{tft} \Big( ~~~ $}
  \begin{tikzpicture}
  \shadedraw[color=gray] (0,\locpa) circle (\locpa);
  \draw[very thick,color=blue] (0,0) -- (0,2*\locpa);
  \node (i) at (0.23,0.24) {$\scs x_i $};
  \end{tikzpicture}
  \raisebox{3.3em} {$ ~~~ \Big) $}
  \label{eq:|i>>}
  \ee
of Ishibashi states yields the expression \eqref{|x_a>}.\,%
 \footnote{~The precise normalization depends in fact on the conventions for the
 two-point functions of bulk fields on the sphere; see e.g.\ \Cite{Sect.\,4.3}{fffs3}.}

Now when evaluating the invariant \erf{eq:|x_a>}, the ribbon link appearing in the picture 
is interpreted as a morphism in the category \C. Moreover, with the help of the 
duality structure on \C\ we can bend down the $i$-line in the so obtained morphism
according to
  \be
  \begin{tikzpicture}
  \draw[very thick,color=red] (-0.08*\locpa,1.35*\locpa) arc (95:445:0.85*\locpa cm and 0.35*\locpa cm);
  \draw[very thick,color=blue] (0,0) -- (0,0.58*\locpa);
  \draw[very thick,color=blue] (0,0.74*\locpa) -- (0,1.7*\locpa);
  \node (a) at (0.48*\locpa,0.57*\locpa) {$\scs x_a $};
  \node (i) at (0.24,0.09) {$\scs x_i^{} $};
  \end{tikzpicture}
  \raisebox{3.3em} {\hspace*{2.0em} $ \xmapsto{~~~~} $ \hspace*{1.5em} }
  \begin{tikzpicture}
  \draw[very thick,color=red] (-0.41*\locpa,0.77*\locpa) arc (222:570:0.85*\locpa cm and 0.35*\locpa cm);
  \draw[very thick,color=blue] (0,0) -- (0,0.58*\locpa);
  \draw[very thick,color=blue] (0,0.74*\locpa) .. controls (0,1.35*\locpa)
                               and (-0.8,1.35*\locpa) .. (-0.8,0) ;
  \node (a) at (0.65*\locpa,0.57*\locpa) {$\scs x_a $};
  \node (i) at (0.24,0.09) {$\scs x_i^{} $};
  \node (ii) at (-1.07,0.09) {$\scs x_i^\vee $};
  \end{tikzpicture}
  \label{eq:openHopf2char}
  \ee
Upon summation over $i \iN I$, the morphism on the right hand side of
\erf{eq:openHopf2char} is indeed precisely the character $\chii^L_{x_a}$ of a simple
$L$-module $(x_a,\rho_a)$ with $\rho_a$ a canonical action of $L$ on the object $x_a \iN \C$.

\subsection{Boundary states in finite CFT}

A crucial observation is now that by making use of the coend structure of $L$ the result just
described for rational CFT actually generalizes directly to non-semisimple finite CFTs. A
detailed justification of this statement will be given in Section \ref{sec:chii-cochi}.
Let us point out that the characters of $L$-modules appearing here are not to be confused 
with characters in the sense of vertex operator algebras. However, as will be explicated
in Remark \ref{rem:chiv}, they indeed directly correspond to chiral genus-1 one-point
functions for vertex algebra representations.

Thus the TFT construction of correlators of rational CFTs precisely yields a standard 
mathematical structure that is still present for arbitrary finite conformal field theories
-- a lattice $K_0(\C) \,{\hookrightarrow}\, \EndID \,{\cong}\, \HomC(L,\one)$. 
Moreover, as a generic feature of decategorification, this lattice comes with the additional 
structure of a distinguished basis -- in our case, the characters $\chii^L_{x_a}$ of the
simple $L$-modules $(x_a,\rho_a)$. For non-semisimple \C\ the lattice is not of maximal 
rank, i.e.\ the characters $\chii^L_{x_a}$ do not span the whole space $\HomC(L,\one)$. 

To allow for an interpretation of this result in CFT terms, we need to identify the
vector space $\HomC(L,\one)$ with a space of conformal blocks. The space of conformal
blocks in question is not the one of zero-point blocks on the torus (which is also
isomorphic to $\HomC(L,\one)$ \cite{lyub6}), but the one for a disk with one bulk
field insertion. As will be explained in Section \ref{sec:bdyblocks}, this follows by
combining recent developments \cite{fScS3} concerning conformal blocks for surfaces
with boundary with an appropriate expression for the space of bulk fields in the Cardy
case. More specifically, we need to describe the latter as an object, and in fact even
as a commutative Frobenius algebra, in the category \CbC, i.e.\ in the Deligne product of
\C\ with its reverse \Cb. (\Cb\ is the same category as \C, but with reversed braiding
and twist, which accounts for the opposite chirality of left- and right-movers.)
For rational CFTs, the space of bulk fields of the Cardy case is realized by the object
  \be
  \bigoplus_{i \in I}\, x_i^\vee \boxtimes x_i^{} \,\in \CbC \,,
  \label{eq:bulkalgebra-ssi}
  \ee
which in particular gives rise to the charge-conjugate (sometimes also called diagonal)
torus partition function.

We need to generalize this expression to arbitrary finite CFTs.
We do so by the following further natural hypothesis about the Cardy case:
\\[3pt]
{\bf (F)}~\,
We assume that for any finite CFT the bulk object in the Cardy case is the coend
  \be
  \Fc:= \int^{c\in\C}\! c^\vee \,{\boxtimes}\, c  \,\in \CbC \,.
  \label{eq:Fc}
  \ee

The object $\Fc$ combines left- and right-movers in the same way as in rational CFT:
it pairs each object with its charge-conjugate, modulo dividing out all morphisms
between objects. 
When \C\ is semisimple, this leaves one representative out of each isomorphism class of
simple objects, so that $\Fc$ reduces to the Cardy bulk algebra \erf{eq:bulkalgebra-ssi}.

\medskip

The assumption {\bf (F)} about the bulk object is logically independent from the 
assumptions {\bf (BC)} and {\bf (BS)} about boundary conditions and boundary states 
made above. It is remarkable that
  \begin{enumerate}
  \item
by the proper notion of modularity of braided finite tensor categories, $\Fc \iN \CbC$
gives an object $F$ in the Drinfeld center \ZC\ of \C\ that is a commutative symmetric
Frobenius algebra in \ZC, whereby also $\Fc$ naturally is such an algebra 
(see Section \ref{sec:bulkalgebra});
  \item
with this Frobenius algebra in \ZC, the conformal blocks for the correlator of one bulk 
field on the disk can be shown to be canonically isomorphic to the center $\HomC(L,\one)$
(see Section \ref{sec:bdyblocks}).
  \end{enumerate}

Boundary states must satisfy a number of consistency requirements. Most notably, upon
sewing they must lead to annulus amplitudes that in the open-string channel can be
expanded in terms of characters. In the non-semisimple case this is a 
non-trivial requirement, as it excludes contributions from so-called pseudo-characters.
Moreover, the coefficients in such an expansion must be non-negative integers, as befits a
partition function of open string states or boundary fields. It is then a further remarkable
observation of our paper that the setup laid out above furnishes consistent annulus
partition functions, with coefficients taking values in the positive integer cone.

\bigskip

The format of the paper is as follows: We start in Section \ref{sec:2} by presenting 
pertinent results about (not necessarily semisimple) modular tensor categories and about
algebraic structures internal to them, in particular the Frobenius algebra structure on
the coend bulk object \erf{eq:Fc} and characters and cocharacters of $L$-modules.
Important input needed for this description has become available only recently
\cite{shimi7,shimi9,shimi10,fScS2} and has not been adapted to the CFT setting before. 
Taking \C, as a module category over itself, as the category of boundary conditions,
in Section 3 we then obtain the spaces of conformal blocks for incoming
and outgoing boundary states and present, in Postulates \ref{pstl:in} and \ref{pstl:out}, 
our precise proposal for the boundary states. Afterwards, in Section
4, these boundary states are used to obtain, via sewing, annulus amplitudes. We show that
the open-string channel annulus amplitudes can be expressed as non-negative integral
linear combinations of characters, so that they can be consistently be interpreted as
partition functions. 
As a further consistency check, we show that the annulus amplitudes are compatible with 
the natural proposal that the boundary fields can be described as internal Hom objects for
\C\ as a module category over itself. We therefore conjecture that the boundary operator
products can be expressed through the structure maps of these internal Homs.
The Appendix provides additional information about some of the
mathematical tools that are used in the main text.

As an illustration, for three specific classes of models -- the rational case, the logarithmic
$(p,1)$ triplet models \cite{kaus,gaKa3,fgst} whose boundary states have been studied in
\cite{garu,garu2,gaTi}, and the case that \C\ is the \rep\ category of a finite-dimensional
factorizable ribbon Hopf algebra -- we present further details about the Cardy case bulk 
object (Example \ref{X:1}), the spaces of conformal blocks for boundary states 
(Example \ref{X:2}) and their subspaces spanned by (co)characters (Example \ref{X:3}),
and finally the annulus amplitudes obtained from them by sewing (Example \ref{X:4}).

%%%%%%%%%%%%%%%%%%%%%%%%%%%%%%%%%%%%%%%%%%%%%%%%%%%%%%%%%%%%%%%%%%%%%%%%

\section{Structures in modular tensor categories}\label{sec:2}

\subsection{Modular tensor categories beyond semisimplicity}

In this section we present the class of categories relevant to us. We also survey
pertinent structure in these categories, to be used freely in the rest of the paper.
For applications to conformal field theory, the categories in question should be thought
of as (being ribbon equivalent to) representation categories of appropriate
$C_2$-cofinite vertex operator algebras.
All categories considered in this paper will be \complex-linear.\,%
 \footnote{~The statements below remain true if \complex\ is replaced by any 
 algebraically closed field $\ko$. In the CFT application, $\ko \eq \complex$.}

\begin{Definition}[Finite category]
A \complex-linear category \C\ is called \emph{finite} iff
\\[3pt]
{\rm 1.}\, \C\ has finite-dimensional spaces of morphisms;
\\[3pt]
{\rm 2.}\, every object of \C\ has finite length;
\\[3pt]
{\rm 3.}\, \C\ has enough projective objects;
\\[3pt]
{\rm 4.}\, there are finitely many isomorphism classes of simple objects.
\end{Definition}

\begin{rem}
A \complex-linear category is finite if and only if it is equivalent
to the category $A$-mod of \findim\ modules over a \findim\ \complex-algebra $A$.
\end{rem}

\begin{Definition}[Finite tensor category]
A \emph{finite tensor category} is a rigid monoidal finite
\complex-linear category with simple tensor unit.
\end{Definition}

The tensor product $\otimes$ in a finite tensor category is automatically exact in each
argument. As a consequence, the Grothendieck group $K_0(\C)$ of \C\ inherits a ring
structure; it is thus referred to as the Gro\-then\-dieck ring, or \emph{fusion ring}, of
\C. A semisimple finite tensor category is also called a \emph{fusion category}.
Without loss of generality we take the monoidal structure to be \emph{strict}, i.e.\ assume
that the tensor product is strictly associative and that the monoidal unit $\one$ 
obeys $c \oti \one \eq c \eq \one \oti c$ for all objects $c$.

The categories of our interest are not only monoidal, i.e.\ endowed with a tensor product,
and rigid, i.e.\ endowed with left and right dualities, but have further structure:
they are also braided and have a twist, or balancing, satisfying compatibility relations
which correspond to properties of ribbons embedded into three-space. (For more details
see e.g.\ Chapter XIV of \cite{KAss} and Section 2.1 of \cite{fuRs4}).

\begin{Definition}[Ribbon category]
A \emph{ribbon category} (or \emph{tortile} category) is a balanced braided rigid monoidal category.
\end{Definition}

A ribbon category is in particular endowed with a canonical \emph{pivotal} structure,
i.e.\ a choice of monoidal natural isomorphism between the (left or right)
double dual functor and the identity functor or, equivalently, with a \emph{sovereign}
structure, i.e.\ a choice of monoidal natural isomorphism between the left and the right
dual functors, The pivotal structure can be expressed through the twist together with the
dualities and braiding or, conversely, the twist through the pivotal structure together with the
dualities and braiding. A ribbon category is also \emph{spherical}, i.e.\ the left and right 
trace of any endomorphism are equal.

We denote the right and left dual of an object $c$ of a rigid category by $c^\vee$ and 
${}^{\vee\!}c$, respectively, and the corresponding evaluation and coevaluation morphisms 
by $\ev_c\colon c^\vee \oti c \To \one$ and $\coev_c \colon \one \To c \oti c^\vee$, and by
$\tev_c\colon c\oti {}^{\vee\!}c \To \one$ and $\tcoev_c\colon \one \To {}^{\vee\!}c \oti c$,
respectively. For the braiding between objects $c$ and $d$ of a braided category we write
$\cb_{c,d}\colon c\oti d \xcong d\oti c$, and for the twist on an object $c$ of a ribbon
category we write $\theta_c\colon c \xcong c$. Henceforth we will,  
for the sake of brevity, often tacitly identify each object of a ribbon category with its 
double dual, i.e.\ suppress the pivotal structure, since it can be restored unambiguously.

\medskip

A \emph{coend} is a specific colimit that, morally, amounts to summing over all objects
of a category while at the same time dividing out all relations among them that
are implied by morphisms in the category. It vastly 
generalizes the direct sum $\bigoplus_{i\in I}$, as appearing e.g.\ in the expression
\eqref{eq:bulkalgebra-ssi} for the bulk state space in the semisimple case, to which it
reduces if the category is finitely semisimple (for more details see e.g.\ \cite{fuSc23}).
Coends are defined through a universal property; thus if a coend exists, it is unique up to
unique isomorphism.
In any finite tensor category \C\ and for any object $c\iN\C$ the exactness of the
tensor product and of the duality functor guarantees that the specific coend
  \be
  Z(c) := \int^{x\in\C}\! x^\vee \oti c \oti x
  \label{eq:Z(c)}
  \ee
exists as an object in \C, see Theorem 3.4\,%
   \footnote{~In the preprint version of \cite{shimi7}, this is Theorem 3.6.}
of \cite{shimi7}. If \C\ is semisimple, then $Z(c)$
is a finite direct sum $\bigoplus_{i\in I} x_i^\vee \oti c \oti x_i^{}$.
The prescription \erf{eq:Z(c)} defines an endofunctor $Z$ of \C\ that admits an algebra
structure, i.e.\ is a \emph{monad}, even a Hopf monad, on \C\ (for details see Appendix 
\ref{app:centralmonad}). Of particular interest is the object
  \be
  \Ze = \int^{x\in \C}\! x^\vee\otimes x ~\in \C \,,
  \label{eq:efL}
  \ee
which has a natural structure of algebra in \C.
To appreciate these statements, note that the notions of an algebra and coalgebra, as
well as Frobenius algebra, can be defined in any monoidal category \C\ (e.g., in any
category of endofunctors) in full analogy with the category of vector spaces. When
\C\ is in addition braided, the same holds for the notion of a Hopf algebra.
  
If the finite category \C\ is braided, then the object $\Ze$ in fact has a canonical
structure of a Hopf algebra in \C, and the Hopf monad $Z$ can be obtained by tensoring
with $\Ze$. When regarding $\Ze$ as a Hopf algebra we write 
  \be
  \Ze =: L \,.
  \ee
We denote the multiplication, unit, comultiplication, counit and antipode of $L$ by
$\mulL \,{\equiv}\, \mul_L$, $\etaL \,{\equiv}\, \eta_L$, $\DeltaL \,{\equiv}\, \Delta_L$, 
$\epsL \,{\equiv}\, \eps_L$ and $\apoL \,{\equiv}\, \apo_L$, respectively.
The Hopf algebra $L$ also comes with a natural pairing
  \be
  \omegaL :\quad L\otimes L \to \one \,,
  \label{eq:hopa}
  \ee
which has the structure of a \emph{Hopf pairing}, i.e.\ satisfies the compatibility relations
  \be
  \bearll
  \omegaL \circ (\mulL \oti \id_L) = \omegaL \circ \big[ \id_L \otimes
  \big( (\omegaL \oti \id_L) \cir (\id_L \oti \DeltaL) \big) \big] \,, ~~
  & \omegaL \circ (\etaL \oti \id_L) = \epsL \,,
  \Nxl3
  \omegaL \circ (\id_L \oti \mulL) = \omegaL \circ \big[
  \big( (\id_L \oti \omegaL ) \cir (\DeltaL \oti \id_L) \big) \otimes \id_L \big] \,,
  & \omegaL \circ (\id_L \oti \etaL) = \epsL 
  \eear
  \ee
with the structural morphisms of $L$ as a bialgebra.

As a coend, $L \eq \Ze$ comes with a morphism $\iL_x \colon x^\vee \oti x \To L$ for each
$x\iN\C$, forming a \emph{dinatural} family.
This implies that a morphism $f\colon L \To c$ is uniquely determined by the dinatural family
$f\cir \iL_x \colon x^\vee \oti x \To c$ of morphisms, and likewise for morphisms with source
object $L\oti L$ 
etc. For instance, the following formula determines the Hopf pairing $\omegaL$ uniquely:
  \be
  \omegaL \circ ( \iL_x \oti \iL_y ) = (\ev_x \oti \ev_y) \circ
  \big( \id_{x^\vee_{}} \oti ( \cb_{y^\vee_{},x} \cir \cb_{x,y^\vee_{}}) \oti \id_y \big) \,.
  \ee
Analogously, the defining formulas for the structural morphisms of the Hopf algebra $L$ read
  \be
  \bearll
  \mulL \circ (\iL_x \oti \iL_y) := \iL_{y\otimes x}
  \circ (\id_{x^\vee_{}} \oti \cb_{x,y^\vee_{}\otimes y}) \,,
  & \etaL := \iL_\one \,,
  \Nxl3
  \DeltaL \circ \iL_x
  := (\iL_x\oti \iL_x) \circ (\id_{x^\vee_{}} \oti \coev_x \oti \id_x) \,, \qquad 
  & \epsL \circ \iL_x :=€\ev_x ,
  \Nxl3 \multicolumn2l {
  \apoL \circ \iL_x := (\ev_x \oti \iL_{x^\vee_{}})
  \circ (\id_{x^\vee_{}} \oti \cb_{x^{\vee\!\vee}_{}\!,x} \oti \id_{x^\vee_{}})
  \circ (\coev_{x^\vee_{}} \oti \cb_{x^\vee_{}\!,x}) \,. }
  \eear
  \label{eq:L-morphisms}
  \ee
Here in the formula for $\mulL$ the trivial identifications of
$\id_{x^\vee_{}} \oti \id_{y^\vee_{}}$ with $\id_{(y\otimes x)^\vee_{}}$
and of $\id_{y} \oti \id_{x}$ with $\id_{y\otimes x}$ are implicit.
A graphical interpretation of the expressions for the coproduct $\DeltaL$ and the counit
$\epsL$ looks as follows:
  \def\locpa  {0.45}  
  \def\locpb  {0.35}  
  \def\locpc  {0.7}   
  \def\locpd  {1.0}   
  \def\locph  {0.1}   
  \def\locpk  {0.9}   
  \def\locpl  {0.7}   
  \def\locpL  {0.4}   
  \def\locpm  {0.6}   
  \def\locpM  {1.2}   
  \def\locpw  {0.4}   
  \def\locpz  {0.5}   
  \be
  \begin{tikzpicture}
  \drawDinatStdWLegs{0}{0} {\scs \iL_x} {\locpl}
  \draw[very thick,color=\Lcolor] (0,0) -- (0,\locpc);
  \draw[very thick,color=\Lcolor] (-\locpw,\locpc+\locpw) arc (180:360:\locpw cm);
  \filldraw[color=blue!60!black] (0,\locpc) circle (0.09) node[above]{$\scs \DeltaL$};
  \draw[very thick,color=\Lcolor] (-\locpw,\locpc+\locpw) -- +(0,\locpm) node[above]{$L$};
  \draw[very thick,color=\Lcolor] (\locpw,\locpc+\locpw) -- +(0,\locpm) node[above]{$L$};
  \node at (1.5,0.5) {$=$};
  \begin{scope}[shift={(3.5,\locpz)}]
  \drawDinatStdTleftWLegs{-\locpa}{0} {\scs \iL_x~~~} {\locpL}
  \draw[very thick,color=\Lcolor] (-\locpa,0) -- +(0,\locpM) node[above]{$L$};
  \drawDinatStdWLegs{\locpa}{0} {\scs \iL_x} {\locpL}
  \draw[very thick,color=\Lcolor] (\locpa,0) -- +(0,\locpM) node[above]{$L$};
  \draw[very thick,color=\xxcolor] (\locph-\locpa,-\locpL-0.18) arc (180:261:\locpb cm);
  \draw[very thick,color=\xxcolor,->] (\locpa-\locph,-\locpL-0.18) arc (360:261:\locpb cm);
  \draw[very thick,color=\xxcolor] (-\locpa-\locph,-\locpL) -- +(0,-\locpd);
  \draw[very thick,color=\xxcolor] (\locpa+\locph,-\locpL) -- +(0,-\locpd);
  \end{scope}
  \begin{scope}[shift={(7.6,0.6)}]
  \drawDinatStdTleftWLegs{0}{0} {\scs \iL_x~~~} {\locpk}
  \draw[very thick,color=\Lcolor] (0,0) -- (0,\locpc);
  \fill[color=blue!50,draw=black,thick] (0,\locpc) circle (0.09) node[right]{$\scs \epsL$};
  \end{scope}
  \node at (8.6,0.5) {$=$};
  \begin{scope}[shift={(9.3,0.5)}]
  \draw[very thick,color=\xxcolor] (0,0) arc (180:102:\locpb cm);
  \draw[very thick,color=\xxcolor,->] (2*\locpb,0) arc (0:102:\locpb cm);
  \draw[very thick,color=\xxcolor] (0,0) -- +(0,-\locpk);
  \draw[very thick,color=\xxcolor] (2*\locpb,0) -- +(0,-\locpk) node[below]{$\scs x$};
  \end{scope}
  \end{tikzpicture}
  \ee
The corresponding description of the product $\mulL$ will be provided in formula \erf{eq:mF} 
below. (The picture for the antipode will not be needed; it can e.g.\ be found in 
\Cite{Eq.\,(4.19)}{fuSc17}.)

\medskip

To any monoidal category \A\ there is canonically associated a braided monoidal category
$\ZA$, called the \emph{monoidal center}, or \emph{Drinfeld center}, of \A. The objects of 
$\ZA$ are pairs $(a,\cB)$ consisting of an object $a\iN\A$ and a natural family 
$\cB \eq (\cB_b)_{b\in\A}$ of isomorphisms $\cB_b\colon a \oti b \xcong b \oti a$
satisfying one half of the properties of a braiding and accordingly called a half-braiding.\,%
 \footnote{~We follow the convention used e.g.\ in \cite{dmno}; in \cite{EGno}
 the half-braiding is defined in the opposite manner.}
The associativity constraint of $\ZA$ is the same as the one of $\C$ (and thus in the 
present context, by strictness, taken to be trivial), while the braiding of $\ZC$ is 
(see e.g.\ \Cite{Prop.\,8.5.1}{EGno})
  \be
  \cza_{(a,\cB),(a',\cB')} = \cB_{a'} \,.
  \ee
Forgetting the half-braiding furnishes an exact monoidal functor 
  \be
  U:\quad \Z(\A)\to \A \,. 
  \label{eq:Uforget}
  \ee

If \C\ has a (right, say) duality, then so has its Drinfeld center \ZC, with the same 
evaluation and coevaluation morphisms $\ev_a$ and $\coev_a$ as in \C\ and with dual objects
  \be
  (a,\cB)^\vee = (a^\vee,\cB^-)§ \,,
  \ee
where
  \be
  (\cB^-)_b^{} := (\ev_a\oti \id_b \oti \id_{a^\vee})
  \circ (\id_{a^\vee} \oti \cB_b^{-1} \oti \id_{a^\vee})
  \circ  (\id_{a^\vee} \oti \id_b \oti \coev_a)
  \ee
is the partial dualization of the inverse half-braiding of $a$. In particular, if, as in the
case of our interest, \C\ is ribbon, then \ZC\ naturally comes with a ribbon structure.\,%
 \footnote{~Recall that we suppress the pivotal structure of \C. Likewise we suppress the
 pivotal structure for \ZC. This is consistent because \Cite{Exc.\,7.13.6}{EGno} 
 a pivotal structure of \C\ induces one for \ZC.}
Henceforth we usually reserve the symbol \C\ for braided categories; given a braided category \C,
we write $\cb$ for the braiding in \C\ and $\cz$ for the braiding in its center \ZC.

\medskip

We denote by \Cb\ the \emph{reverse} of a finite ribbon category \C, i.e.\ the same monoidal
category, but with inverse braiding and twist.  For any finite ribbon category \C\ there
is a canonical braided functor
  \be
  \GC: \quad \Cb \boti \C \to \Z(\C)
  \label{eq:CbC->ZC}
  \ee
from the \emph{enveloping category} of \C, i.e.\ the Deligne product of \Cb\ with \C,
to the Drinfeld center of \C.
As a functor, \GC\ maps the object $\Ol u \boti v \iN \CbC$ to the tensor product 
$\Ol u \oti v \iN \C$ endowed with the half-braiding $\cB_{\Ol u\otimes v}$ that has components
  \be
  \cB_{\Ol u\otimes v;c} = ( \cb_{c,\Ol u}^{-1} \oti \id_v ) \circ ( \id_{\Ol u} \oti \cb_{v,c} ) 
  \label{eq:cB}
  \ee
for $c \iN \C$. We will freely use the graphical calculus for morphisms in 
the braided monoidal category \C. We then have the following graphical description of the 
half-braiding \erf{eq:cB}:
  \be
  \begin{tikzpicture}
  \braid[ line width=\piclinesize, height=0.8cm,
     style strands={1}{red},
     style strands={2,3}{\locco}
  ] (braid) at (0,0) s_1 s_2^{-1};
  \node[at=(braid-rev-1-e),below=\piclabetsep] {$u$};  
  \node[at=(braid-rev-2-e),below=\piclabetsep] {$v$}; 
  \node[at=(braid-rev-3-e),below=\piclabetsep] {$c$}; 
  \node (b1)  at (0.45,-0.95) {$\ssg \cb^{-1} $};
  \node (b2)  at (1.46,-1.73) {$\ssg \cb $};
  \node (lhs) at (-1.6,-1.4)  {$ \cB_{\Ol u\otimes v;c} ~= $};
  \end{tikzpicture}
  \label{eq:cz_uv,c}
  \ee
with the individual braidings in the picture being braidings in \C.
The braided monoidal structure on the functor \GC\ is given by the coherent family
$\id_{\Ol u} \oti \cb_{v,\Ol x} \oti \id_y$ of isomorphisms from
$\Ol u \oti v \oti \Ol x \oti y$ to $\Ol u \oti \Ol x \oti v \oti y$
(for details see Appendix \ref{more:GC}).

\medskip
 
In the theories relevant to us the braiding obeys a
non-degeneracy condition. This condition can be formulated in several equivalent ways: 

\begin{Definition}[Modular tensor category]\label{def:modular}
A \emph{modular tensor category} is a finite ribbon category \C\ which satisfies
one of the following equivalent \cite{shimi10} conditions:
  \Itemize
  \item 
The canonical functor \GC\ \eqref{eq:CbC->ZC} is a braided equivalence.
  \item 
The Hopf pairing $\omegaL$ \eqref{eq:hopa} on the coend $L \iN \C$ is non-degenerate.
  \item 
The linear map $\HomC(\one,L)\To \HomC(L,\one)$ that is induced by
the Hopf pairing $\omegaL$ is an isomorphism of vector spaces.
  \item
The category $\C$ has no non-trivial transparent objects, i.e.\ any object
having trivial monodromy with every object is a finite direct sum
of copies of the tensor unit $\one$.
  \end{itemize}
\end{Definition}

Modularity of \C\ is crucial for constructing a modular functor in the sense of \cite{lyub6}.

\begin{Example}
(i)\,~Let us stress that for \C\ being modular it is \emph{not} required that it is
semisimple. If \C\ \emph{is} in addition semisimple, corresponding to the case of rational
CFTs, then the conditions in Definition \ref{def:modular} are equivalent \cite{brug2,muge6}
to the familiar requirement that the modular S-matrix is non-degenerate. 
 \\
In the semisimple case, each indecomposable object is simple, so that in
particular up to isomorphism there are finitely many indecomposable objects. For
instance, in the case of the category $\C_{\mathfrak g,\ell}$ that is relevant for the
WZW model based on a semisimple Lie algebra $\mathfrak g$ and positive integer
$\ell$, the isomorphism classes of indecomposable objects are labeled by the finitely many 
integrable highest weights of the untwisted affine Lie algebra
$\mathfrak g^{\scriptscriptstyle(1)}$ at level $\ell$.
\\[2pt]
(ii)\,~Classes of logarithmic conformal field theories for which the representation
category of the chiral algebra is known to be a (non-semisimple) modular tensor category 
are the symplectic fermion models \cite{abe4,daRu,fagR2} and the $(p,1)$ triplet models
\cite{adMi3,tswo}.
While, like in the semisimple case,
the corresponding categories have finitely many simple objects up to isomorphism,
already for the arguably simplest logarithmic CFT, the $(2,1)$ triplet model, there are
uncountably many isomorphism classes of indecomposable objects \cite{fgst2}.
\\[2pt]
(iii)\,~The category $H$\Mod\ of \findim\ modules over any \findim\ factorizable ribbon 
Hopf (or, more generally, weak quasi-Hopf) algebra $H$ is a modular tensor category. It 
is semisimple iff $H$ is semisimple as an algebra.
\end{Example}

%%%%%%%%%%%%%%%%%%%%%%%%%%%%%%%%%%%%%%%%%%%%%%%%%%%%%%%%%%%%%%%%%%%%%%%%

\subsection{The Cardy bulk algebra}\label{sec:bulkalgebra}

The property of a braided finite tensor category of being modular has important 
consequences. Two of these are of particular interest to us:
First, recall that in a full local conformal field theory the chiral degrees of freedom
of left and right movers taken together are described in terms of the enveloping category
$\Cb\boti\C$. In the case of a modular tensor category we can replace $\Cb\boti\C$ by the
Drinfeld center $\Z(\C)$. Working with $\Z(\C)$ allows us in the semisimple case to use 
the conformal blocks of the topological field theory of Turaev-Viro type that is
associated with \C\ (see e.g.\ \cite{balKi}). For Turaev-Viro type theories the obstruction
in the Witt group to the existence of boundary conditions \cite{fusV} vanishes, so that
such conformal blocks are also available for surfaces with boundary.
(The boundary Wilson lines are then labeled by \C\ itself, as befits a Cardy case.)

Second, a modular tensor category is in particular \emph{unimodular} \Cite{Prop.\,4.5}{etno2}.
It follows (see Theorems 4.10 and 5.6 of \cite{shimi7})
that the forgetful functor $U$ from \ZC\ to \C\ in \eqref{eq:Uforget} has a two-sided adjoint
$I\colon\, \C\To\Z(\C)$, i.e.\ is a Frobenius functor. (There are then corresponding
Frobenius algebras in \C\ and \ZC, and associated 
with them \Cite{Thm.\,8.2}{karv} Frobenius monads on \C\ and \ZC.)
As a result the object 
  \be
  F := I(\one) ~\in \Z(\C)
  \label{eq:F=I(1)}
  \ee
has a natural structure of a commutative Frobenius algebra in the braided tensor category $\Z(\C)$.
When transported to the enveloping category \CbC, the object \erf{eq:F=I(1)} is nothing but 
the coend $\Fc$ introduced in \erf{eq:Fc}, i.e.\ we have $\GC(\Fc) \eq F$. As explained there,
this object is expected to furnish the bulk state space for the Cardy case; its algebra 
structure provides the bulk field operator products, while the non-degenerate Frobenius form 
encodes the two-point correlation function of bulk fields on the sphere. To constitute the
bulk object of a consistent full CFT, $F$ must not only be a commutative symmetric Frobenius
algebra, but in addition also modular \cite{fuSc22}; this has so far only been fully 
established for the cases that the category \C\ is semisimple (see \Cite{Thm.\,3.4}{koRu2} 
and \cite{fuSc22}) or that it is the \rep\ category of a Hopf algebra \Cite{Cor.\,5.11}{fuSs3}.

\begin{Example}\label{X:1}
(i)\,~In the finitely semisimple case the Cardy case bulk algebra as an object of \CbC\ is 
the direct sum $\Fc_{\rm fss} \eq \bigoplus_{i\in I} x_i^\vee \boti x_i^{}$ as given in 
\eqref{eq:bulkalgebra-ssi}, while the Frobenius algebra $F \iN \Z(\C)$ can be written as the object
  \be
  F_{\rm fss} = \bigoplus_{i\in I}\, x_i^\vee \otimes x_i^{}
  \label{eq:Fss=sum_i..}
  \ee
in \C\ together with the half-braiding described explicitly e.g.\ in \Cite{Thm.\,2.3}{balKi}.
The Frobenius algebra structure (on $\Fc_{\rm fss}$) is given in \Cite{Lemma\,6.19}{ffrs}.
\\[2pt]
(ii)\,~It follows from Corollary 5.1.8 of \cite{KEly} that the coend 
$\Fc \eq \int^{c\in \C} c^\vee \boti c \iN \CbC$ can be written as
  \be
  \Fc \cong \big( P \boti P \big) / N \,,
  \ee
where $P \eq \bigoplus_{i\in I} P_i$, the direct sum of (representatives for the isomorphism 
classes of) all indecomposable projective objects of \C, is a projective generator, and
$N$ is the subspace obtained by acting with $f \boti \id_P - \id_P \boti f$ for all
$f\iN\EndC(P)$. A representation-theoretic description of $N$ has been given in
\Cite{Sect.\,3.3}{garu2} for a class of models that includes in particular 
the logarithmic $(p,1)$ triplet models, as the kernel of
a pairing defined in terms of three-point conformal blocks with one insertion from the
the vertex algebra itself. In the particular case of the $(2,1)$ triplet model (of Virasoro
central charge $-2$) which describes symplectic fermions, this kernel can be expressed
in terms of the zero mode of the chiral fermion field \Cite{Eq.\,(2.11)}{garu}.
 \\
It is also known that for the $(p,1)$ models the class $[F]$ in the Grothendieck ring
is given by $\sum_{i\in I} [x_i \oti P_i]$, with $x_i^{} \,{\cong}\, x_i^\vee$
the simple objects and $P_i$ the indecomposable projectives (which are
the projective covers of the $x_i$) \Cite{Sect.\,4.4}{garu2} or, what is the same
(see e.g.\ \cite{fuSs4}), by $\sum_{i,j\in I} C_{ij}\, [x_i \oti x_j]$ with 
$(C_{ij})$ the Cartan matrix of the category.
For comparison, the Cartan matrix of a semisimple category is the identity matrix,
so that $[F_{\rm fss}] \eq \sum_{i\in I} [x_i^\vee \oti x_i^{}]$, in agreement with
\eqref{eq:Fss=sum_i..}.
\\[2pt]
(iii)\,~For $\C \eq H$\Mod\ the category of \findim\ left modules over a \findim\ factorizable 
ribbon Hopf algebra $H$, the enveloping category \CbC\ is braided equivalent to the category
of \findim\ $H$-bimodules \Cite{App.\,A.2}{fuSs3}. The coend $\Fc$ is then the dual vector
space $H^*$ endowed with the co-regular left and right $H$-actions \Cite{App.\,A.1}{fuSs3},
while $\Ze \iN H\Mod$ is $H^*$ with the co-adjoint left $H$-action \Cite{Sect.\,4.5}{vire4}.
\end{Example}

It should, however, be appreciated that a decomposition into a direct sum of factorized 
objects, as for $\Fc_{\rm fss}$, no longer occurs when \C\ is non-semisimple, not even in 
the Cardy case, in which $\Fc$ is instead given by \erf{eq:Fc}. For a general, not 
necessarily semisimple, modular tensor category,
the Cardy case bulk algebra in \CbC\ is the coend $\Fc$ as given in in \erf{eq:Fc},
while in \ZC\ it is the object $F$ consisting of the object in \C\ given by the coend
  \be
  U(F) = \Ze = \int^{c\in \C}\! c^\vee\otimes c \,.
  \label{eq:UFL}
  \ee
together with a half-braiding.
In particular, the object $\Ze$ of \C, besides having a structure of Hopf algebra in \C,
also naturally comes with a half-braiding $\gamma$ such that $(\Ze,\gamma) \eq F \iN \ZC$
has a structure of a symmetric commutative Frobenius algebra in \ZC, and this Frobenius
structure is unique up to a scalar (see \Cite{Lemma\,3.5}{dmno} and \Cite{Thm.\,5.6}{shimi7}).
This half-braiding can be obtained explicitly by realizing that 
$Z(\Ze) \,{\cong}\, \Ze \oti \Ze$ and using the dinatural morphisms for $Z(\Ze)$ as a 
coend \erf{eq:Z(c)} together with the product $\mul_F$. It is, however, best described with 
the help of the \emph{monad} structure on the endofunctor $c \,{\mapsto}\, Z(c)$ and
realizing that modules over the monad $Z$ are the same as objects with a half-braiding. The 
object $Z(c)$ has a canonical structure of a $Z$-module.

\medskip

We will need to know the Frobenius algebra structure on the object $F$ of $\ZC$ explicitly.
Let us first describe the algebra structure on $F$. The relevant commutative associative
multiplication on $\Fc \iN \CbC$ has been described in \Cite{Prop.\,2.3}{fuSs6}. We need 
to transport this product along the functor \GC. When doing so we must account for the
monoidal structure on \GC: For $H\colon \D \To \mathcal D'$ a tensor functor with monoidal
structure $\varphi$ and $A$ an algebra in $\mathcal D$ with product $m$, the corresponding
product $m'$ on the algebra $H(A)$ is the composition 
$H(m) \cir \varphi_{A,A}\colon H(A) \oti H(A) \To H(A \oti A) \To H(A)$.
As shown in Lemma \ref{lem:braidedstructureonGC}, the monoidal structure on \GC\
is given by a braiding; we then find the following description of the multiplication
morphism $\mul_F \eq \GC(\mul_{\Fc}) \cir \varphi_{\Fc,\Fc}$ in \ZC:
\def\lochc  {0.3}   
\def\locwb  {0.4}   
  \be
  \begin{tikzpicture}
  \braid[ line width=\piclinesize, height=0.6cm, width=\locwb cm,
     style strands={1,2,3,4}{\xxcolor}
  ] (braid) at (0,0) s_3^{-1} s_2^{-1};
  \filldraw[fill=yellow!40, draw=\xxcolor, thick]
     (-0.3*\locwb,0) -- (1.3*\locwb,0) -- (1.3*\locwb,\lochc) -- (-0.3*\locwb,\lochc) -- cycle ;
  \filldraw[fill=yellow!40, draw=\xxcolor, thick]
     (1.7*\locwb,0) -- (3.3*\locwb,0) -- (3.3*\locwb,\lochc) -- (1.7*\locwb,\lochc) -- cycle ;
  \draw[very thick,color=\xxcolor] (0.5*\locwb,\lochc) .. controls (0.5*\locwb,\lochc+0.6) and
                                   (1.5*\locwb-0.2,\lochc+0.3) .. (1.5*\locwb-0.2,\lochc+0.9) ;
  \draw[very thick,color=\xxcolor] (2.5*\locwb,\lochc) .. controls (2.5*\locwb,\lochc+0.6) and
                                   (1.5*\locwb+0.2,\lochc+0.3) .. (1.5*\locwb+0.2,\lochc+0.9) ;
  \drawDinatWLegs {1.5*\locwb} {\lochc+1.1} {0.4} {0.3} {\ssg \iL_{y\otimes x}} {0}
  \draw[very thick,color=\Lcolor] (1.5*\locwb,\lochc+1.1) -- (1.5*\locwb,\lochc+1.8) ;
  \node (lhs) at (-2.7,-.2) {$ \mul_F^{} \circ (\iL_x \oti \iL_y) ~= $};
  \node[at=(braid-rev-1-e),below=\piclabetsep] {$\ssg x^\vee $};
  \node[at=(braid-rev-2-e),below=\piclabetsep] {$\ssg ~x^{\phantom\vee} $};
  \node[at=(braid-rev-3-e),below=\piclabetsep] {$\ssg y^\vee $};
  \node[at=(braid-rev-4-e),below=\piclabetsep] {$\ssg ~y^{\phantom\vee} $};
  \node (F) at (1.5*\locwb,\lochc+2.08) {$\ssg \Ze $};)
  \end{tikzpicture}
  \label{eq:mF}
  \ee
Here the two un-labeled coupons stand for the (trivial) identifications
$\id_{x^\vee_{}} \oti \id_{y^\vee_{}} \eq \id_{(y\otimes x)^\vee_{}}$
and $\id_{y} \oti \id_{x} \eq \id_{y\otimes x}$, respectively.

\begin{lem}~%
{\rm (i)}\, The morphism $\mul_F^{}$ in \ZC\ defined by \eqref{eq:mF} is a commutative 
associative multiplication for $F$.
\\[2pt]
{\rm (ii)}\, The morphism $\eta_F^{} \,{:=}\, \iZe_\one$ is a unit for the product $\mul_F^{}$.
\end{lem}

\begin{proof}
(i)\, Associativity is guaranteed by the fact that $\mul_F$ is obtained
from an associative product for $\Fc$. It can also be verified directly through an
exercise in braid gymnastics,
which pictorially looks as follows:
 \def\locPutDinatWLegsEtcOnTop{%
  \filldraw[fill=yellow!40, draw=\xxcolor, thick]
     (-0.3*\locwb,0) -- (2.3*\locwb,0) -- (2.3*\locwb,\lochc) -- (-0.3*\locwb,\lochc) -- cycle
     (2.7*\locwb,0) -- (5.3*\locwb,0) -- (5.3*\locwb,\lochc) -- (2.7*\locwb,\lochc) -- cycle;
  \drawDinatWLegs {2.5*\locwb} {\lochc+1.1} {0.4} {0.3} {} {0}
  \draw[very thick,color=\Lcolor] (2.5*\locwb,\lochc+1.1) -- +(0,0.8) ;
  \draw[very thick,color=\xxcolor] (\locwb,\lochc) .. controls (\locwb,\lochc+0.6) and
                                   (2.5*\locwb-0.2,\lochc+0.3) .. (2.5*\locwb-0.2,\lochc+0.9) ;
  \draw[very thick,color=\xxcolor] (4*\locwb,\lochc) .. controls (4*\locwb,\lochc+0.6) and
                                   (2.5*\locwb+0.2,\lochc+0.3) .. (2.5*\locwb+0.2,\lochc+0.9) ;
  }
  \def\locph  {3.7}   
  \def\locpe  {2.86}  
  \def\locpv  {1.2}   
  \def\locpw  {0.6}   
  \def\locpx  {4.3}   
  \be
  \begin{tikzpicture}
  \node at (-4.05,-\locpw) {$\mul_F\cir(\mul_F\oti\id_F)$};
  \node at (-2.8,-\locpv) {$\cir(\iL_x{\otimes}\iL_y{\otimes}\iL_z)~=$};
  \node at (\locpe,-\locpv) {$=$};
  \node at (\locpe+\locph,-\locpv) {$=$};
  \node at (6.1,-\locpx) {$=~\mul_F\cir(\id_F\oti\mul_F)\cir(\iL_x{\otimes}\iL_y{\otimes}\iL_z)\,.$};
  \braid[ line width=\piclinesize, height=0.6cm, width=\locwb cm,
     style strands={1,2,3,4,5,6}{\xxcolor}
     ] (braid) at (0,0) s_4^{-1} s_3^{-1}-s_5^{-1} s_4^{-1} s_3^{-1} s_2^{-1};
  \locPutDinatWLegsEtcOnTop
  \drawDinatWLegs {2.5*\locwb} {\lochc+1.1} {0.4} {0.3} {\ssg \iL_{z\otimes y\otimes x}} {0}
  \node[at=(braid-rev-1-e),below=\piclabetsep] {$\ssg x^\vee $};
  \node[at=(braid-rev-2-e),below=\piclabetsep] {$\ssg ~x^{\phantom\vee} $};
  \node[at=(braid-rev-3-e),below=\piclabetsep] {$\ssg y^\vee $};
  \node[at=(braid-rev-4-e),below=\piclabetsep] {$\ssg ~y^{\phantom\vee} $};
  \node[at=(braid-rev-5-e),below=\piclabetsep] {$\ssg z^\vee $};
  \node[at=(braid-rev-6-e),below=\piclabetsep] {$\ssg ~z^{\phantom\vee} $};
  \begin{scope}[shift={(\locph,0)}]
  \braid[ line width=\piclinesize, height=0.6cm, width=\locwb cm,
     style strands={1,2,3,4,5,6}{\xxcolor}
     ] (braid) at (0,0) s_4^{-1} s_5^{-1} s_4^{-1} s_3^{-1} s_2^{-1}-s_4^{-1};
  \locPutDinatWLegsEtcOnTop
  \begin{scope}[shift={(\locph,0)}]
  \braid[ line width=\piclinesize, height=0.6cm, width=\locwb cm,
     style strands={1,2,3,4,5,6}{\xxcolor}
     ] (braid) at (0,0) s_5^{-1} s_4^{-1} s_3^{-1} s_2^{-1}-s_5^{-1} s_4^{-1};
  \locPutDinatWLegsEtcOnTop
  \end{scope}
  \end{scope}
  \end{tikzpicture}
  \ee
Commutativity is seen as follows. With the help of the dinatural family $\iL$, the
braiding $\cz_{F,F}$ can be expressed in terms of the braidings 
$\cz_{x^\vee_{}\otimes x,y^\vee_{}\otimes y}$ such that
  \be
  \mul_F \circ \cz_{F,F} \circ (\iL_x \oti \iL_y)
  = \mul_F \circ (\iL_y \oti \iL_x)  \circ \cz_{x^\vee_{}\otimes x,y^\vee_{}\otimes y} \,.
  \ee
Moreover, dinaturality of $\iL$ implies the identity
  \be
  \iL_{y\otimes x} \circ (\id_{(y\otimes x)^\vee_{}} \oti \cb_{x,y})
  = \iL_{x\otimes y} \circ ( \cb_{x^\vee_{},y^\vee_{}} \oti \id_{x\otimes y}) \,.
  \ee
Combining these equalities with the definition \erf{eq:mF} of $\mul_F$, commutativity boils
down to the relation
  \be
  \bearl
  ( \cb_{y^\vee_{},x^\vee_{}} \oti \id_{x\otimes y}) 
  \circ (\id_{y^\vee_{}} \oti \cb_{y,x^\vee_{}} \oti \id_x)
  \circ \cz_{x^\vee_{}\otimes x,y^\vee_{}\otimes y} 
  \Nxl2  \hspace*{11em}
  = (\id_{x^\vee_{}\otimes y^\vee_{}} \oti \cb_{x,y})
  \circ (\id_{x^\vee_{}} \oti \cb_{x,y^\vee_{}} \oti \id_y) \,.
  \eear
  \ee
After inserting the expression \erf{eq:cz_uv,c} for $\cz$, this reduces to the identity
 \def\lochb  {0.5}   
 \def\lochB  {1.25}  
 \def\locwb  {0.5}   
  \be
  \begin{tikzpicture}
  \braid[ line width=\piclinesize, height=\lochb cm, width=\locwb cm,
  style strands={1,2,3,4}{\xxcolor}
  ] (braid) at (0,0) s_1^{-1} s_2^{-1} s_2 s_1-s_3^{-1} s_2^{-1};
  \node[at=(braid-rev-1-e),below=\piclabetsep] {$\ssg x^\vee $};
  \node[at=(braid-rev-2-e),below=\piclabetsep] {$\ssg ~x^{\phantom\vee} $};
  \node[at=(braid-rev-3-e),below=\piclabetsep] {$\ssg y^\vee $};
  \node[at=(braid-rev-4-e),below=\piclabetsep] {$\ssg ~y^{\phantom\vee} $};
  \node[blue!70!black,rotate=90] at (-0.48,-1.97)
                                 {$\scs \cz_{x^\vee_{}\otimes x,y^\vee_{}\otimes y}$};
  \draw[blue!70!black] (-0.2,-2.5*\lochb) rectangle (0.2+3*\locwb,-5.5*\lochb);
  \node at (2.63,-1.6) {$=$};
  \braid[ line width=\piclinesize, height=\lochB cm, width=\locwb cm,
  style strands={1,2,3,4}{\xxcolor}
  ] (braid) at (3.6,0) s_3^{-1} s_2^{-1};
  \node[at=(braid-rev-1-e),below=\piclabetsep] {$\ssg x^\vee $};
  \node[at=(braid-rev-2-e),below=\piclabetsep] {$\ssg ~x^{\phantom\vee} $};
  \node[at=(braid-rev-3-e),below=\piclabetsep] {$\ssg y^\vee $};
  \node[at=(braid-rev-4-e),below=\piclabetsep] {$\ssg ~y^{\phantom\vee} $};
  \end{tikzpicture}
  \ee
among braids and is thus indeed satisfied.
\nxl1
(ii) By definition of the product one has $\mul_F \cir (\iL_\one \oti \iL_x) \eq 
 \iL_x \eq \mul_F \cir (\iL_x \oti \iL_\one)$ for all $x \iN \C$, i.e.\ $\iL_\one$
satisfies the properties of a unit for $\mul_F$.
\end{proof}

\medskip

It is worth noting that the product and unit of $F \iN \ZC$ are the same, when regarded 
as morphisms in \C, as those of the Hopf algebra $L \eq \Ze \eq U(F) \iN \C$,
  \be
  \eta_F = \etaL \qquand \mul_F = \mulL \,.
  \label{etaF=etaL,mulF=mulL}
  \ee
This should in fact not come as a surprise. Indeed, the antipode of the Hopf algebra $L$ is
invertible, and $H$ admits a left integral $\Lambda \iN \HomC(\one,H)$ and a right cointegral 
$\lambda  \iN \HomC(H,\one)$ such that $\lambda \cir \Lambda \iN \EndC(\one)$ is invertible.
By definition, a left integral $\Lambda$ of $H$ is a morphism in $\HomC(\one,H)$ that
intertwines the trivial and regular left $H$-actions, i.e.\ satisfies
  \be
  \LambdaL \circ \epsL = \mulL \circ (\id_L \oti \LambdaL) \,.
  \ee
Similarly, a right cointegral $\lambda$ obeys by definition
  \be
  \etaL \circ \lambdaL = (\lambdaL \oti \id_l) \circ \DeltaL \,.
  \ee
Since the category \C\ is unimodular, so is the Hopf algebra $L$, i.e.\ the left integral is
also a right integral, i.e.\ satisfies $\LambdaL \cir \epsL \eq \mulL \circ (\LambdaL \oti \id_l)$,
implying in particular that the integral is invariant under the antipode,
$  \apoL \cir \LambdaL \eq \LambdaL $.  
Now it is known \Cite{App.\,A.2}{fuSc17} that in a linear ribbon category any Hopf algebra 
with invertible antipode and (co)integrals for which $\lambda \cir \Lambda$ is invertible,
also carries a natural Frobenius algebra structure, with the same product and unit,
but with different coalgebra structure.

\medskip

Moreover, the Hopf algebra structure on the object $L \eq U(F)$ in \C\ together with the
integral and cointegral of $L$ also induces a coalgebra structure on $U(F)$ that is part
of its Frobenius algebra structure in \C\ \Cite{App.\,A.2}{fuSc17}. In view of 
\erf{etaF=etaL,mulF=mulL} one would expect that also the coalgebra structure on $F$ is 
such that upon forgetting the half-braiding on $F$ it reproduces this Frobenius coalgebra 
structure on $U(F)$ in \C. This indeed turns out \Cite{Sect.\,5}{shimi9} to be the case. In
particular, the Frobenius counit is given by the cointegral of $L$.
This conclusion also coincides with the expectations from CFT.
Indeed, the counit $\eps_F$ is provided by the one-point correlator of bulk fields 
on the sphere \cite{fuSc22}, and by comparison with the one-point correlators for 
semisimple \C\ one expects that, as a morphism in \C, it is a (non-zero) cointegral 
of the Hopf algebra $L$,
  \be
  \epsF = \lambdaL \,.
  \ee
Actually, this is already implied by the fact \Cite{Sect.\,5.1}{shimi9} that the 
relevant morphism space is one-dimensional:
  \be
  \HomZ(F,\one_\ZC) = \HomZ(I(\one),\one_\ZC)
  \,{\cong}\,\, \HomC(\one,U(\one_\ZC)) = \HomC(\one,\one) \cong \koc \,.
  \ee

Given the counit, the \emph{Frobenius form} $\omegaF\colon F \oti F \To \one$ is given by 
  \be
  \omegaF = \epsF^{} \circ \mul_F^{} = \lambdaL \circ \mulL \,.
  \label{eq:omegaF}
  \ee
This is indeed non-degenerate, as required for a Frobenius form; a side inverse
$\omegaFm$, satisfying $(\omegaF \oti \id_F) \cir (\id_F \oti \omegaFm) \eq \id_F
\eq (\id_F \oti \omegaF) \cir (\omegaFm \oti \id_F)$ is given by
(see e.g. \Cite{Eq.\,(3.33)}{fuSc17})
  \be
  \omegaFm = (\id_L \oti \apoL) \circ \DeltaL \circ \LambdaL
  \label{eq:omegaFm}
  \ee
with $\apoL$ and $\DeltaL$ the antipode and coproduct of $L$, and $\LambdaL$ an 
integral of $L$ satisfying 
  \be
  \lambdaL \circ \LambdaL = 1 \,.
  \label{lambdaL.LambdaL=1}
  \ee

As for any Frobenius algebra \cite{fuSt}, the Frobenius coproduct can be obtained from 
the product $\mul_F \eq \mulL$ by appropriately composing with 
$\omegaF$ and its side inverse. Concretely, with the help of the isomorphism
  \be
  \Phi := (\omegaF \oti \id_{F^\vee_{}}) \circ (\id_F \oti \coev_F) ~\in \HomC(F,F^\vee)
  \label{eq:Phi}
  \ee
we can write (compare \Cite{Eq.\,(5.5)}{shimi9})
  \be
  \bearll
  \Delta_F \!\!& 
  = (\id_F \oti \mulF) \circ (\id_F \oti \Phi^{-1} \oti \id_F) \circ (\coev_F \oti \id_F) 
  \Nxl2&
  = (\mulF\oti \id_F) \circ (\id_F \oti \id_F \oti \Phi^{-1}) \circ (\id_F \oti \coev_F) \,.
  \eear
  \label{eq:DeltaF}
  \ee
By using the explicit expression
  \be
  \Phi^{-1} = (\ev_F \oti \id_F) \circ (\id_{F^\vee_{}} \oti \omegaFm)
  \label{eq:Phi-1}
  \ee
for the inverse of \eqref{eq:Phi}, this can be rewritten as 
  \be
  \bearll
  \Delta_F \!\!&
  = (\id_F \oti \mulF) \circ (\id_F \oti \apoL \oti \id_F) \circ
  \big( (\DeltaL \cir \LambdaL) \oti \id_F \big)
  \Nxl2&
  = (\mulF \oti \apoL) \circ \big( \id_F \oti (\DeltaL \cir \LambdaL) \big)
  = (\mulF \oti \id_F) \circ ( \id_F \oti \omegaFm ) \,, \eear
  \label{eq:DeltaF=A10}
  \ee
thereby reproducing formula (A.10) of \cite{fuSc17}.
A mirror version of this formula is valid as well; graphically, the two formulas look like
  \be
  \begin{tikzpicture}
  \node at (-2.5,0.25) {$\DeltaF ~=$};
  \draw[very thick,color=\Lcolor] (0,0) -- (0,0.5) arc (0:180:0.6 cm) -- (-1.2,-1);
  \draw[very thick,color=\Lcolor] (-0.6,1.1) -- +(0,0.7);
  \draw[very thick,color=\Lcolor] (0.6,0) -- +(0,1.8);
  \filldraw[color=red!60!black] (-0.6,1.1) circle (0.11) node[below=0.6pt]{$\scs \mulF$};
  \drawomegaFm {0.3} {0}
  \node at (2.12,0.25) {$=$};
  \begin{scope}[shift={(3.6,0)}]
  \draw[very thick,color=\Lcolor] (0,0) -- +(0,1.8);
  \draw[very thick,color=\Lcolor] (0.6,0) -- (0.6,0.5) arc (180:0:0.6 cm) -- (1.8,-1);
  \draw[very thick,color=\Lcolor] (1.2,1.1) -- +(0,0.7);
  \filldraw[color=red!60!black] (1.2,1.1) circle (0.11) node[below=0.6pt]{$\scs \mulF$};
  \drawomegaFm {0.3} {0}
  \end{scope}
  \end{tikzpicture}
  \ee
It follows e.g.\ that the integral satisfies (compare \Cite{Prop.\,5.5}{shimi9})
$ \Lambda \eq \big( \id_F \oti (\eps_L \cir \Phi^{-1}) \big) \cir \coev_F$.  

\medskip

Let us also show that $F \iN \ZC$ has trivial twist, By definition, the twist on the 
ribbon category \CbC\ is given by $\theta^{-1} \boti \theta$ with $\theta$
the twist on \C. Accordingly the following statement is not surprising:

\begin{lem}\label{lem:twistZC}
For $u,v\iN\C$ the twist of the object $\GC(u\boti v) \iN \ZC$ can be expressed as
  \be
  \theta^{\ZC}_{\GC(u\boxtimes v)} = \theta_u^{-1} \oti \theta_v^{}
  \label{eq:thetaZC}
  \ee
in terms of the twist $\theta$ in \C.
\end{lem}

\begin{proof}
Expressing the twist of the ribbon category \ZC\ through the braiding and dualities
and using the explicit form \eqref{eq:cB} for the braiding of $\GC(u\boti v)$ in \ZC,
we have, pictorially:
\def\lochb  {0.5}  
\def\locwb  {0.6}  
\def\locpw  {3*\locwb}  
\def\locph  {2.0}  
\def\locpH  {1.5}  
\def\locphi {0.7}  
\def\locpwi {0.9}  
\def\locpwj {1.1*\locwb}  
\def\locpho {1.7}  
\def\locpso {2.0*\locwb}  
\def\locpwo {2.1}  
\def\locpx  {6.1}  
\def\locpxx {8.1}  
\def\locpy {-0.3}  
\def\locle  {1.6}  
  \be
  \begin{tikzpicture}
  \node (lhs) at (-1.7,-1.0) {$ \theta^{\ZC}_{\GC(u\boxtimes v)} ~~= $};
  \braid[ line width=\piclinesize, height=\lochb cm, width=\locwb cm, style strands={1,2,3,4}\locco,
  ] (braid) at (0,0) s_2 s_1-s_3^{-1} s_2^{-1};
  \draw[very thick,color=\locco] { (0,0) -- (0,\locle) };
  \draw[very thick,color=\locco] { (\locwb,0) -- (\locwb,\locle) };
  \draw[very thick,color=\locco] { (0,-\locph) -- (0,-\locph-\locle) };
  \draw[very thick,color=\locco] { (\locwb,-\locph) -- (\locwb,-\locph-\locle) };
  \node (u1) at (0,-\locph-\locle-.21) {$\ssg u $};
  \node (v1) at (\locwb,-\locph-\locle-.21) {$\ssg v $};
  \drawCloseToTwist \locpw 0 \locphi \locpwi \locph \locco
  \drawCloseToTwist \locpso 0 \locpho \locpwo \locph \locco
  \node (eq) at (4.6,-1.0) {$ = $};
  \braid[ line width=\piclinesize, width=\locwb cm, style strands={1,2}\locco,
  ] (braid) at (\locpx,\locpy) s_1;
  \braid[ line width=\piclinesize, width=\locwb cm, style strands={1,2}\locco,
  ] (braid) at (\locpxx,\locpy) s_1^{-1};
  \draw[very thick,color=\locco] { (\locpx,\locpy) -- (\locpx,\locle) };
  \draw[very thick,color=\locco] { (\locpxx,\locpy) -- (\locpxx,\locle) };
  \draw[very thick,color=\locco] { (\locpx,\locpy-\locpH) -- (\locpx,-\locph-\locle) };
  \draw[very thick,color=\locco] { (\locpxx,\locpy-\locpH) -- (\locpxx,-\locph-\locle) };
  \node (u2) at (\locpx,-\locph-\locle-.21) {$\ssg u $};
  \node (v2) at (\locpxx,-\locph-\locle-.21) {$\ssg v $};
  \drawCloseToTwist {\locpx+\locwb} \locpy \locphi \locpwj \locpH \locco
  \drawCloseToTwist {\locpxx+\locwb} \locpy \locphi \locpwj \locpH \locco 
  \end{tikzpicture}
  \ee
\end{proof}

\begin{lem}\label{lem:FinZC}
{\rm (i)}\, $F$ has trivial twist in \ZC.
\\[2pt]
{\rm (ii)}\, $F$ is a symmetric Frobenius algebra in \ZC.
\end{lem}

\begin{proof}
(i)\, By the naturality of the twist we have
$ \theta_F^{\ZC} \cir \iF_x \eq \iF_x \cir \theta^{\ZC}_{x^\vee_{}\otimes x} $, 
where on the right hand side $x^\vee \oti x$ is, by construction, the object
$\GC(x^\vee{\boxtimes}x)$ of \ZC. Lemma \ref{lem:twistZC} thus implies
  $  
  \theta_F^{\ZC} \cir \iF_x
  \eq \iF_x \cir \big( \theta^{-1}_{x^\vee_{}} \oti \theta^{}_x \big) .
  $  
By dinaturality of $\iF$ this becomes
  \be
  \theta_F^{\ZC} \circ \iF_x = \iF_x \circ \big( \id^{}_{x^\vee_{}} \otimes 
  (\theta^{-1}_x \cir \theta^{}_x) \big) = \iF_x \equiv \id_F \cir \iF_x \,.
  \ee
This holds for every $x \iN \C$; the claim thus follows by dinaturality.
\\[2pt]
(ii)\, That $F$ is Frobenius is a direct consequence of the construction of the
coalgebra structure from the Hopf algebra structure and (co-)integral of $L$.
That the Frobenius form \erf{eq:omegaF} is symmetric follows by combining the
facts that the product $\mul_F$ is commutative and that $F$ has trivial twist.
\end{proof}
\begin{rem}
In the case of semisimple \C, in which $F \eq \bigoplus_{i\in I} x_i^\vee \oti x_i^{}$,
Lemma \ref{lem:FinZC}(i) follows immediately by combining \eqref{eq:thetaZC} with
$\theta_{x^\vee}^{} \eq \theta_x^{}$. Lemma \ref{lem:FinZC}(ii) follows in this case 
from the fact that, by Lemma 6.19(ii) of \cite{ffrs}, $\Fc \iN \CbC$ is symmetric
Frobenius. (By the same lemma, for semisimple \C\ the Frobenius algebra $F$ is also
special; this ceases to be true for non-semisimple \C.)
\end{rem}

By construction, the morphisms $\mulF$, $\etaF$, $\DeltaF$ and $\epsF$ also endow,
when regarded as morphisms in \C, the object $U(F) \eq Z(\one) \iN \C$ with
the structure of a Frobenius algebra.
However, since the braiding and twist in \C\ are different from those in \ZC,
unlike $F\eq I(\one) \iN \ZC$ this Frobenius algebra is \emph{not} commutative, and $U(F)$ 
does not have trivial twist. (The existence of a two-sided adjunction between \C\ and \ZC\ 
does not require \C\ to be braided, so the lack of commutativity is not so surprising.)
Moreover, even though $U(F)$ is a symmetric algebra, the Frobenius form $\omegaF$ is not 
symmetric on the nose when regarded as a morphism in \C, but there is a minor deviation from
being symmetric.  Indeed, combining the expression \erf{eq:omegaF} for $\omegaF$ with general
properties of the cointegral, and with the fact that the square of the antipode $\apoL$ of 
$L$ is an inner automorphism, one concludes (compare \Cite{Thm.\,3}{radf12}) that
  \be
  \omegaF \circ \cb_{L,L} = \omegaF \circ (\id_L \oti \apoL^{-2}) \,.
  \label{omegaF.cLL}
  \ee
In other words, $\apoL^2$ is an inner Nakayama automorphism for the Frobenius algebra $U(F)$.
Going from \C\ to \ZC\ changes the braiding and requires a trivial Nakayama automorphism.
(Note that this means that $S_L^4$ is the identity in \ZC. This might be expected
to be the case for a consistent full CFT even beyond finite CFTs, because the
anomaly in the action of the modular group that is generically present
\cite{lyub6} on the chiral level should cancel out between left and right movers.)

%%%%%%%%%%%%%%%%%%%%%%%%%%%%%%%%%%%%%%%%%%%%%%%%%%%%%%%%%%%%%%%%%%%%%%%%

\subsection{Characters and cocharacters}\label{sec:chii-cochi}

As outlined in the Introduction, the boundary states for the Cardy case furnish the 
standard mathematical structure of a ring homomorphism from the fusion ring
$K_0(\C)$ to the endomorphisms of the identity functor of \C. We now construct such
a homomorphism concretely with the help of the Hopf algebra $L$.  As ingredients
we need the notions of algebras and their modules in rigid monoidal categories.

The representation theory of (associative, unital) algebras in a rigid monoidal category
can be studied in full analogy with the classical case of algebras over \complex. In
particular, characters are defined as partial traces in the same way as (see e.g.\
\Cite{Sect.\,1.5}{loren}) for a \complex-algebra. (Below we will only be interested in the 
case that \C\ is ribbon and thus spherical, so that we do not have to distinguish
between left and right traces.)
In more detail, for $A \iN \C$ an algebra, an $A$-\emph{module} $(m,\rho_m)$ is an object 
$m \iN \C$ together with a representation morphism $\rho_m\colon A \oti m \To m$ that 
satisfies the usual compatibility requirements with the product and unit of $A$.

\begin{Definition}[Character]
The \emph{character} of an $A$-module $(m,\rho_m)$ in \C\ is the partial trace
  \be
  \chii_m^A := \tilde \ev_m \circ (\rho_m \oti \id_{m^\vee_{}}) \circ (\id_A \oti \coev_m) 
  ~\in \HomC(A,\one)
  \ee
of $\rho_m$.
\end{Definition}

Here $\tev_x\colon x \oti x^\vee \To \one$ is the (left) evaluation morphism,
and $\coev_x\colon \one \To x \oti x^\vee$ is the (right) coevaluation. Pictorially,
\def\locph  {0.7}  
\def\locpw  {0.45} 
\def\locpx  {0.4}  
\def\locpy  {0.6}  
\be
  \begin{tikzpicture}
  \node (lhs) at (-1.0,1.1) {$ \chii_m^A ~= $};
  \draw[very thick,color=\locco] (\locpx,\locpy+\locpw) -- (\locpx,\locpy+\locpw+\locph);
  \draw[very thick,color=\locco] (\locpx+2*\locpw,\locpy+\locpw) -- (\locpx+2*\locpw,\locpy+\locpw+\locph);
  \draw[very thick,color=\locco] (\locpx+2*\locpw,\locpy+\locpw) arc (360:180:\locpw cm);
  \draw[very thick,color=\locco,->] (\locpx+2*\locpw,\locpy+\locpw) arc (360:263:\locpw cm);
  \draw[very thick,color=\locco] (\locpx,\locpy+\locpw+\locph) arc (180:0:\locpw cm);
  \draw[very thick,color=\locco,->] (\locpx,\locpy+\locpw+\locph) arc (180:83:\locpw cm);
  \filldraw[fill=\locco!50!white,draw=black] (\locpx-0.1,\locpy+\locpw+\locph/2-0.1) rectangle (\locpx+0.1,\locpy+\locpw+\locph/2+0.1);
  \draw[very thick] (0,0) .. controls (0,\locpy+\locpw+\locph/2-0.1) and
                          (\locpx-0.08,\locpy) .. (\locpx-0.08,\locpy+\locpw+\locph/2-0.1);
  \node[anchor=north] at (0,0) {$\scs A $};
  \node at (\locpx-0.25,\locpy+\locpw+\locph/2+0.4) {$\scs m $};
  \end{tikzpicture}
  \ee
Characters split under extensions, i.e.\ for any short exact sequence 
  \be
  0\to m' \to m \to m'' \to 0 
  \ee
of $A$-modules we have 
  \be
  \chii^A_m = \chii^A_{m'} + \chii^A_{m''} \,.
  \label{eq:split}
  \ee
As a consequence, the character of any module is a sum of irreducible characters,
i.e.\ of characters of simple modules. Isomorphic modules have equal characters.

\medskip

{}From now on we assume that \C\ is braided. If $A$ is even a bialgebra, then one can form
tensor products of representations: The tensor product of two $A$-modules $(m,\rho^A_m)$
and $(n,\rho^A_n)$ is the object $m \oti n$ together with the \rep\ morphism
  \be
  \rho^A_{m \otimes n} := (\rho_m \oti \rho_n) \circ (\id_A \oti \cb_{A,m} \oti \id_n)
  \circ (\Delta_A \oti \id_m \oti \id_n) \,,
  \label{eq:rho_mn=rho_m.rho_n}
  \ee
with $\Delta_A$ the coproduct of $A$. Characters satisfy
  \be
  \chii^A_{m \otimes n} = (\chii^A_m \oti \chii^A_n) \circ \Delta_A \,,
  \label{eq:chii_mn=chii_m.chii_n}
  \ee
i.e.\ they are multiplicative with respect to the convolution product on $\HomC(A,\one)$
that is induced by the coalgebra structure of $A$ (and the algebra structure of $\one$).

In the case of the Hopf algebra $L \eq \Ze$, every object $m$ in \C\ turns out to have
a canonical structure of an $L$-module $(m,\rho^L_m)$. 
The action $\rho^L_m\colon L \oti m \To m$ of $L$ on $m$ is defined via a partial 
monodromy. More explicitly, $\rho^L_m$ can be expressed with the help of the
dinatural family of morphisms $\imath^L_x\colon x^\vee \oti x \To L$ that comes with
the coend structure of $L$, and of the braiding $\cb$ and the right evaluation $\ev$ 
of \C, as \Cite{Sect.\,2.2}{fuSs5}
  \be
  \rho^L_m \circ (\iL_x \oti \id_m) := (\ev_x \oti \id_m) \circ \big[
  \id_ {x^\vee_{}} \oti ( \cb_{m,x} \cir \cb_{x,m} ) \big]
  \label{eq:rhoLm}
  \ee
for all $x \iN \C$.
Pictorially, this action and the resulting character are given by
  \be
  \begin{tikzpicture}
  \braid[ line width=\piclinesize, height=0.8cm, width=0.5cm,
  style strands={1,2}{\xxcolor},
  style strands={3}{\locco}
  ] (braid) at (0,0) s_2^{-1} s_2^{-1};
  \drawRightEval {0.25} {0.25} {0.25} 0 \xxcolor ;
  \draw[very thick,color=\xxcolor] (0,-2.1) -- (0,-2.4) ;
  \draw[very thick,color=\xxcolor] (0.5,-2.1) -- (0.5,-2.4) ;
  \draw[very thick,color=\locco] (1.0,-2.1) -- (1.0,-2.4) ;
  \draw[very thick,color=\locco] (1.0,0) -- (1.0,0.7) ;
  \node (xv) at (0,-2.4-.21) {$\ssg x^\vee $} ;
  \node (x)  at (0.5,-2.4-.21) {$\ssg x $} ;
  \node (m)  at (1.0,-2.4-.21) {$\ssg m $} ;
  \node (lhs) at (-2.5,-0.9) {$ \rho_m^L \circ (\iL_x \oti \id_m) ~= $} ;
  \end{tikzpicture}
  \hspace*{4em}
  \begin{tikzpicture}
  \braid[ line width=\piclinesize, height=0.6cm, width=0.5cm,
  style strands={1,2}{\xxcolor},
  style strands={3,4}{\locco}            
  ] (braid) at (0,-0.2) s_2^{-1} s_2^{-1};
  \drawRightEval {0.25} {0.25} {0.25} 0 \xxcolor ;
  \drawLeftEval {1.25} {0.05} {0.25} 0 \locco ;
  \drawRightCoeval {1.25} {-1.9-.25} {0.25} 0 \locco ;
  \draw[very thick,color=\xxcolor] (0,0) -- (0,-0.2) ;
  \draw[very thick,color=\xxcolor] (0.5,0) -- (0.5,-0.2) ;
  \draw[very thick,color=\xxcolor] (0,-1.9) -- (0,-2.8) ;
  \draw[very thick,color=\xxcolor] (0.5,-1.9) -- (0.5,-2.8) ;
  \draw[very thick,color=\locco] (1.5,-0.2) -- (1.5,-1.9) ;
  \node (xv) at (0,-2.8-.21) {$\ssg x^\vee $} ;
  \node (x)  at (0.5,-2.8-.21) {$\ssg x $} ;
  \node (m)  at (1.21,-1.66) {$\ssg m $} ;
  \node (lhs) at (-1.8,-1.3) {$ \chii_m^L \circ \iL_x ~= $} ;
  \end{tikzpicture}
  \label{Qrep2} 
  \ee
Notice in particular that for semisimple \C\ the morphism $\chii_m^L \cir \iL_{x_i}$
coincides with the morphism on the right hand side of \erf{eq:openHopf2char}
and that in this case the monodromy appearing in \eqref{Qrep2} gives rise to the
S-matrix factor in the formula \eqref{|x_a>} for the boundary state.
Moreover, any family of characters of pairwise non-isomorphic simple $L$-modules $(m,\rho_m^L)$
is linearly independent in the vector space $\HomC(L,\one)$ \Cite{Thm.\,4.1}{shimi9}.

By \erf{eq:rho_mn=rho_m.rho_n} the $L$-modules $(m,\rho^L_m)$ form a full monoidal subcategory 
of the category $\C_L$ of all $L$-modules in \C. The latter category $\C_L$ is in fact 
braided equivalent to the Drinfeld center \ZC, which is best understood when formulated
\Cite{Thm.\,8.13}{brVi5} in terms of the central monad on \C\ (compare Appendix 
\ref{app:centralmonad}).

\medskip

Through the structure of $L$ as a coend, every object $m$ in \C\ is also canonically 
an $L$-\emph{comodule} $(m,\delta^L_m)$, with the coaction given by
  \be
  \delta^L_m := (\iL_{m^\vee_{}} \oti \id_m) \circ (\id_m \oti \tcoev_m)
  ~\in \HomC(m,L\oti m) \,,
  \ee
and with corresponding \emph{cocharacter}
  \be
  \cochi^L_m := (\id_L \oti \tev_m) \circ (\delta^L_m \oti \id_{m^\vee_{}}) \circ \coev_m
  = \imath^L_{m^\vee_{}} \circ \coev_m ~\in \HomC(\one,L) \,.
  \label{eq:cochi^L_m}
  \ee
Pictorially,  
  \be
  \begin{tikzpicture}
  \drawDinatWLegsLocco 0 0 {0.4} {0.3} {} {0.4}
  \draw[very thick,color=\Lcolor] (0,0) -- (0,1.0) ;
  \draw[very thick,color=\locco] (-0.2,-0.7) -- (-0.2,-1.4) ;
  \drawLeftCoeval {0.2+0.3} {-0.65-0.3} {0.3} 0 \locco
  \draw[very thick,color=\locco] (0.2+0.6,-0.65) -- (0.2+0.6,1.0) ;
  \node (iL) at (-0.69,0.03)  {$\sse \iL_{m^\vee_{}} $} ;
  \node (m)  at (-0.22,-1.61) {$\ssg m $} ;
  \node (lhs) at (-1.9,-0.2)  {$ \delta^L_m ~= $} ;
  \end{tikzpicture}
  \hspace*{4em}
  \begin{tikzpicture} 
  \draw[color=white] (0,-1.8) -- (0,-1.86) ;
  \drawDinatWLegsLocco 0 0 {0.6} {0.3} {} {0.5}
  \draw[very thick,color=\Lcolor] (0,0) -- (0,1.0) ;
  \drawRightCoeval {0} {-0.75-0.3} {0.3} 0 \locco
  \node (mv) at (0.63,-0.49)  {$\ssg m^\vee $} ;
  \node (lhs) at (-1.6,-0.2)  {$ \cochi^L_m ~= $} ;
  \node (nix) at (0,-0.81) {$ $};
  \end{tikzpicture}
  \label{eq:delta,cochi}
  \ee
The character and cocharacter of the trivial $L$-module are given by the counit and
unit of $L$, respectively: $\chii^L_\one \eq \epsL$ and $\cochi^L_\one \eq \etaL$.
Like characters, also cocharacters split under extensions, and they are multiplicative
under the tensor product of comodules, so that \Cite{Sect.\,4.5}{fuSc17}
  \be
  \cochi_{m\otimes n}^L = \mulL \circ (\cochi_m^L \oti \cochi_n^L) \,.
  \label{eq:cochi_mn=cochi_m.cochi_n}
  \ee

The action $\rho_m^L$ and coaction $\delta^L_m$ are connected by the Hopf pairing $\omegaL$
according to $(\omegaL \oti \id_m) \cir 
     $\linebreak[0]$
(\id_L \oti \delta^L_m) \eq \rho^L_m$. 
Recall from Definition \ref{def:modular} that bijectivity of the linear map 
  \be
  \bearll
  \Om : & \HomC(\one,L) \xarr{~~} \HomC(L,\one)
  \Nxl2 & \hspace*{4.2em}
  \tild\alpha \,\xmapsto{~~~}\, \omegaL \cir (\id_L \oti \tild\alpha)
  \eear
  \label{eq:defOm}
  \ee
is one of the equivalent definitions of modularity. Thus for modular \C\ the relation
between the action and coaction can be inverted. In
particular, for the characters and cocharacters we have
  \be
  \chii^L_m = \Om(\cochi^L_m) \qquand \cochi^L_m = \Omm(\chii^L_m) \,. 
  \label{char-cochar}
  \ee

\medskip

The algebra structure of $L$ and the (trivial) coalgebra structure of $\one$ supply a 
\emph{convolution product} on $\HomC(\one,L)$:
  \be
  \tild\alpha \cvp \tild\beta := \mul_L \circ (\tild\alpha \oti \tild\beta)
  \label{tildalpha.tildbeta}
  \ee
for $\tild\alpha,\tild\beta \iN \HomC(\one,L)$. Analogously, the coalgebra structures 
of $L$ and $U(F)$ supply two convolution products on $\HomC(L,\one)$:
  \be
  \alpha \cvpL \beta := (\alpha \oti \beta) \circ \DeltaL \qquad{\rm and}\qquad
  \alpha \cvpF \beta := (\alpha \oti \beta) \circ \Delta_F 
  \ee
for $\alpha,\beta \iN \HomC(L,\one)$.
The multiplicativity relations \erf {eq:chii_mn=chii_m.chii_n}
and \erf{eq:cochi_mn=cochi_m.cochi_n} for the characters and cocharacters tell us
that the (co-)\,character maps respect these products, 
  \be
  \chii^L_{m \otimes n} = \chii_m^L \cvpL \chii_n^L  \qquand
  \cochi_{m\otimes n}^L = \cochi_m^L \cvp \cochi_n^L \,.
  \label{eq:mn=m*n}
  \ee
Since they also split under extensions, we have 

\begin{lem}\label{lem:ringhomom}
The characters $\chii^L_m$ and $\cochi^L_m$ define ring homomorphisms
  \be
  \chii^L:\quad K_0(\C) \xarr{~} \HomC(L,\one) \qquand
  \cochi^L:\quad K_0(\C) \xarr{~} \HomC(\one,L) \,,
  \ee
respectively.
\end{lem}

\begin{rem}
There are \emph{two} distinguished pairs of isomorphisms between the morphism spaces
$\HomC(L,\one)$ and $\HomC(\one,L)$: composition with the \emph{Hopf} pairing $\omegaL$ as
in \erf{eq:defOm} and its side inverse $\omegaLm$ on the one hand, and composition with the
\emph{Frobenius} form \erf{eq:omegaF} and its side inverse $\omegaFm$ on the other hand.
The defining property of the Hopf pairing is equivalent to the first of these pairs of
isomorphisms being intertwiners between the convolution product on $\HomC(\one,L)$ and
the convolution product $\cvpL$ on $\HomC(L,\one)$.
\end{rem}

\begin{rem}
Composing a cocharacter with a character gives a Hopf link morphism. Indeed, by direct
calculation we have
  \be
  \chii^L_y \circ \cochi^L_x
  \stackrel{\erf{char-cochar}}= \omegaL \circ (\cochi^L_x \oti \cochi^L_y) = \soo_{x^\vee,y} \,,
  \label{eq:chii.cochi=soo}
  \ee
with the \emph{Hopf link morphisms} $\soo$ defined as 
  \def\locph  {1.0}  
  \def\locpw  {1.0}  
  \def\locpW  {0.5}  
  \be
  \bearll
  \soo_{x,y} \!\! &:= (\ev_x \oti \tev_y ) \circ \big( \id_{x^\vee} \oti
  [ \cc_{y,x} \cir \cc_{x,y} ] \oti \id_{y^\vee} \big) \circ (\tcoev_x \oti \coev_y) \,.
  \\[49pt] & \,= \hsp{9.9} \in\, \EndC(\one) \,.
  \\[-55pt] & \hsp{2.5}
  \begin{tikzpicture}
  \braid[ line width=\piclinesize, height=0.8cm,
     style strands={1}{red},
     style strands={2}{red},
     style strands={3}{blue}
  ] (braid) at (0,0) s_2 s_2;
  \draw[very thick,color=blue] (3*\locpw,-2*\locph) -- (3*\locpw,0);
  \draw[very thick,color=red] (0,-2*\locph) arc (180:360:\locpW cm);
  \draw[very thick,color=red,->] (0,-2*\locph) arc (180:277:\locpW cm);
  \draw[very thick,color=red] (0,0) arc (180:0:\locpW cm);
  \draw[very thick,color=red,->] (\locpw,0) arc (0:97:\locpW cm);
  \draw[very thick,color=blue] (2*\locpw,-2*\locph) arc (180:360:\locpW cm);
  \draw[very thick,color=blue,->] (3*\locpw,-2*\locph) arc (360:263:\locpW cm);
  \draw[very thick,color=blue] (2*\locpw,0) arc (180:0:\locpW cm);
  \draw[very thick,color=blue,->] (2*\locpw,0) arc (180:83:\locpW cm);
  \node (x) at (\locpw+0.18,-2*\locph) {$\scs x $};
  \node (y) at (2*\locpw-0.21,-2*\locph) {$\scs y $};
  \end{tikzpicture}
  \eear
  \label{eq:def:soo}
  \ee
Using the naturality of the braiding so as to deform the strands in the
picture \erf{eq:def:soo} suitably one shows that
  $  
  \soo_{x,y} \eq \sooi_{y^\vee,x} \eq \soo_{y^\vee,x^\vee}
  $, 
where $\sooi$ is defined analogously as $\soo$, but with the monodromy
$\cc_{y,x} \cir \cc_{x,y}$ replaced by its inverse.

If \C\ is semisimple, then the matrix $(\soo_{i,j})_{i,j\in I}$ is non-degenerate --
it is then just the (unnormalized) modular S-matrix.
In contrast, for non-semisimple \C\ the pairing between the spaces of characters and
cocharacters furnished by $\soo$ is degenerate. Indeed, for any projective object 
$p \iN \C$ and any $x\iN \C$ the morphism $\soo_{p,x} \iN \EndC(\one) \eq \ko$
factorizes through the projective object $p^\vee \oti p \oti x \oti x^\vee$
and can therefore only be non-zero if the tensor unit $\one$ is projective
which, in turn, is equivalent to \C\ being semisimple.
\end{rem}

%%%%%%%%%%%%%%%%%%%%%%%%%%%%%%%%%%%%%%%%%%%%%%%%%%%%%%%%%%%%%%%%%%%%%%%%

\subsection{The modular S-transformation}

If \C\ is a semisimple modular tensor category, then the modular S-transformation of
vertex algebra characters is described, up to normalization, by the modular S-matrix 
$(\soo_{i,j})_{i,j\in I}$ (see e.g.\ \cite{huan20}). In contrast, in the non-semisimple
case this is no longer the case, as is indicated by the degeneracy just mentioned.
However, it is known from \cite{lyub6} that morphism spaces in \C\ of the form
$\HomC(x,L)$ or $\HomC(L,y)$ with $x,y\iN\C$ admit a natural projective action of the 
modular group $\mathrm{SL}(2,\zet)$ by post- or precomposition, respectively, with
endomorphisms of the object $L$. (This generalizes to projective actions of mapping
class groups at any genus $g$ for morphisms involving the object $L^{\otimes g}$.)
This provides us in particular with an $\mathrm{SL}(2,\zet)$-action on the spaces
$\HomC(\one,L)$ or $\HomC(L,\one)$ which contain the cocharacters and characters of
$L$. Specifically, for the modular S-transformation the endomorphism of $L$ in
question is \cite{lyub6}
  \be
  S_L := (\epsL \oti \id_L) \circ \mathcal Q_L \circ (\id_L \oti \LambdaL)
  \ee
with $\mathcal Q_L$ the endomorphism of $L \oti L$ defined by 
  \be
  \mathcal Q_L \circ (\iL_x \oti \iL_y) := (\iL_x \oti \iL_y) \circ \big( \id_{x^\vee}
  \oti [ \cb_{y^\vee,x} \cir \cb_{x,y^\vee} ] \oti \id_y \big) \,.
  \ee
Moreover, the integral $\LambdaL$ can be normalized such that
this morphism squares to the inverse antipode \Cite{Thm\,2.1.9}{lyub6},
  \be
  {(S_L)}^2 = \apoL^{-1} .
  \label{eq:S.S=apoinv}
  \ee
In the sequel we assume that this normalization has been chosen (this determines
$\LambdaL$ up to a sign).

In the case that $\C \eq H\Mod$ is the representation category of a factorizable
ribbon Hopf algebra, $S_L$ is the composition of the Drinfeld and Frobenius
maps of $H$ (see e.g.\ \cite{soZh}). This generalizes as follows:

\begin{lem}
The morphism $S_L$ can be expressed as
  \be
  S_L = (\omegaF \oti \id_L) \circ (\id_L \oti \omegaLm) 
  \label{eq:S=omegaF.omegaLm}
  \ee
in terms of the Hopf
pairing $\omegaL$ \erf{eq:hopa} and of the Frobenius form $\omegaF$ \erf{eq:omegaF}.
\end{lem}
   
\begin{proof}
By direct calculation one sees that
  \be
  \bearll
  (\omegaL \oti \id_L) \circ (\id_L \oti \DeltaL) \circ (\iL_x \oti \iL_y) \!\! &
  = (\coev_x \oti \iL_y) \circ \big( \id_{x^\vee} \oti
  [ \cb_{y^\vee,x} \cir \cb_{x,y^\vee} ] \oti \id_y \big) \,.
  \Nxl3 &
  = (\epsL \oti \id_L) \circ \mathcal Q_L \circ (\iL_x \oti \iL_y) \,,
  \eear
  \label{eq00267}
  \ee
Here the first equality follows by combining the definitions of $\DeltaL$ and $\omegaL$ 
and the second by combining those of $\epsL$ and $\mathcal Q_L$.
By dinaturality, \erf{eq00267} implies that
$ (\omegaL \oti \id_L) \cir (\id_L \oti \DeltaL) \eq (\epsL \oti \id_L) \,{\circ} % \,
  $\linebreak[0]$
  \mathcal Q_L $. 
Pre-composing this equality with $\id_L \oti \LambdaL$ and post-composing with
the antipode gives, by recalling the expression for $\omegaFm$,
  \be
  (\omegaL \oti \id_L) \circ (\id_L \oti \omegaFm)
  = (\eps_L \oti \apoL) \circ \mathcal Q_L \circ (\id_L \oti \LambdaL)
  = \apoL \circ S_L \stackrel{\erf{eq:S.S=apoinv}}= S_L^{-1} \,.
  \ee
Taking inverses, we finally arrive at the claimed formula \erf{eq:S=omegaF.omegaLm}.
\end{proof}

We also have (see e.g.\ \Cite{Lemma\,5.2.4}{KEly})
 \be
  \omegaL \circ \cb_{L,L} = \omegaL \circ (\apoL^{-1} \oti \apoL^{-1})
  \ee
or, equivalently,
  \be
  \cb_{L,L} \circ \omegaLm = (\apo^{}_L \oti \apo^{}_L) \circ \omegaLm .
  \ee
When combined with \erf{omegaF.cLL} and \erf{eq:S=omegaF.omegaLm} it then follows by direct 
calculation
that the formulas relating the S-automorphism to the Hopf pairing and Frobenius form
are left-right-symmetric:

\begin{lem}
The S-endomorphism of $L$ and its inverse obey
  \be
  (\id_L \oti \omegaF) \circ ( \omegaLm \oti \id_L)
  = S_L^{} = (\omegaF \oti \id_L) \circ (\id_L \oti \omegaLm) 
  \label{eq:..=SL=..}
  \ee
and
  \be
  (\id_L \oti \omegaL) \circ ( \omegaFm \oti \id_L)
  = S_L^{-1} = (\omegaL \oti \id_L) \circ (\id_L \oti \omegaFm) \,.
  \ee
\end{lem}

Pictorially,
  \def\locpa  {0.33} 
  \def\locpb  {0.76} 
  \def\locph  {2.7}  
  \def\locpx  {0}    
  \def\locpu  {3.2}  
  \def\locpv  {0.81} 
  \def\locpX  {0.86} 
  \def\locpy  {0}    
  \def\locpY  {1.8}  
  \be
  \begin{tikzpicture}
  \draw[very thick,color=\Lcolor] (\locpa-\locpb,\locpy) -- +(0,\locph);
  \draw[very thick,color=\Lcolor] (\locpb-\locpa,\locpy) -- +(0,\locpY);
  \draw[very thick,color=\Lcolor] (\locpX+\locpb-\locpa,\locpY) -- +(0,-\locph);
  \drawomegaF  \locpX \locpY
  \drawomegaLm \locpx \locpy
  \node at (\locpu,\locpv) {$=~\,S_L^{}\,~=$};
  \begin{scope}[shift={(5.55,0)}]
  \draw[very thick,color=\Lcolor] (\locpa-\locpb,\locpY) -- +(0,-\locph);
  \draw[very thick,color=\Lcolor] (\locpb-\locpa,\locpy) -- +(0,\locpY);
  \draw[very thick,color=\Lcolor] (\locpX+\locpb-\locpa,\locpy) -- +(0,\locph);
  \drawomegaF  \locpx \locpY
  \drawomegaLm \locpX \locpy
  \end{scope}
  \end{tikzpicture}
  \ee
and  
  \def\locpu  {3.36} 
  \be
  \begin{tikzpicture}
  \draw[very thick,color=\Lcolor] (\locpa-\locpb,\locpy) -- +(0,\locph);
  \draw[very thick,color=\Lcolor] (\locpb-\locpa,\locpy) -- +(0,\locpY);
  \draw[very thick,color=\Lcolor] (\locpX+\locpb-\locpa,\locpY) -- +(0,-\locph);
  \drawomegaL  \locpX \locpY
  \drawomegaFm \locpx \locpy
  \node at (\locpu,\locpv) {$=~\,S_L^{-1}\,~=$};
  \begin{scope}[shift={(5.92,0)}]
  \draw[very thick,color=\Lcolor] (\locpa-\locpb,\locpY) -- +(0,-\locph);
  \draw[very thick,color=\Lcolor] (\locpb-\locpa,\locpy) -- +(0,\locpY);
  \draw[very thick,color=\Lcolor] (\locpX+\locpb-\locpa,\locpy) -- +(0,\locph);
  \drawomegaL  \locpx \locpY
  \drawomegaFm \locpX \locpy
  \end{scope}
  \end{tikzpicture}
  \ee

Combining this result with the relation \erf{eq:S.S=apoinv} between $S_L$ and the antipode,
one can derive the following formulas for the inverse of the Hopf pairing:
  \be
  \bearll
  \omegaLm \!\!\! & = (\id_L \oti \omegaL \oti \id_L)
  \circ (\id_L \oti \apoL^{-1} \oti \id_L \oti \id_L) \circ (\omegaFm \oti \omegaFm)
  \Nxl3 & 
  = (\id_L \oti \id_L \oti \omegaL)
  \circ \big( \id_L \oti [\DeltaL\cir\LambdaL] \oti \apoL \big) \circ \DeltaL \circ \LambdaL
  \Nxl3 & 
  = (\id_L \oti \apoL \oti \omegaL)
  \circ \big( \id_L \oti [\DeltaL\cir\LambdaL] \oti \id_L \big) \circ \DeltaL \circ \LambdaL \,.
  \eear
  \ee

Let us also mention that
it follows directly from the defining properties of a Hopf pairing and the non-degeneracy
of $\omegaL$ that if $\Lambda$ is a right integral of $L$, then $\Om(\Lambda)$ is a 
left cointegral, and vice versa, and analogously for left integrals and right cointegrals
(see e.g.\ \Cite{Thm.\,5}{kerl5}).
We can then fix the relative normalization of $\LambdaL$ and $\lambdaL$ in such a way that
  \be
  \lambdaL = \Om(\LambdaL) \circ \apoL^{-1}
  \ee
or, equivalently,
  \be
  \lambdaL = \OmT(\LambdaL) 
  \ee
with the linear map
$\OmT\colon \HomC(\one,L) \To \HomC(L,\one)$ being defined, similarly as $\Om$ in
\erf{eq:defOm}, by $\OmT(\tild\alpha) \,{:=}\, \omegaL \cir (\tild\alpha \oti \id_L)$.

Finally note that
  \be
  \chii_y \circ S_L \circ \cochi_x = \omegaF \circ (\cochi_x \oti \cochi_y) \,,
  \label{eq:chi.S.cochi}
  \ee
which may be compared to \erf{eq:chii.cochi=soo}.

\begin{rem}\label{rem:chiv}
It is tempting to interpret the (co)characters $\chii$ and $\cochi$ in the context of
$C_2$-co\-fi\-nite vertex operator algebras. To do so, note that the appropriate notion
of character
for a module $M$ over a vertex operator algebra \V\ is as a chiral genus-1 one-point function 
  \be
  \chiv_M(v,\tau) = \mathrm{Tr}_M^{}\, o_M^{}(v)\, q^{L_0-c/24} ,
  \label{eq:chiv}
  \ee
with $v\iN \V$ and $o_M(v)$ the grade-preserving endomorphism of $M$ induced by $v$ (rather
than as a Virasoro-specialized character $\chiv_M(\tau)$, which is obtained when taking
$v$ to be the vacuum vector; those are typically not linearly independent).
In conformal field theory terms, these characters span a subspace of the space of one-point 
blocks on the torus with insertion the tensor unit; if the \rep\ category $\V\text{-Rep}$
is non-semisimple, then this is a proper subspace, with a (non-canonical) complement given 
by \cite{miya8} certain \emph{pseudo-characters}. 
(While the characters are power series in $q$, the pseudo-characters involve
extra factors of $\tau \eq \ln q/2\pi\ii$; see e.g.\ \Cite{Eq.\,A.2}{gaTi} for
explicit expressions for the logarithmic $(p,1)$ triplet models.)
In our setting, this space is the morphism
space $\Hom_{\V\text{-Rep}}(\one,L)$ (or the isomorphic space $\Hom_{\V\text{-Rep}}(L,\one)$,
but as we will see below, the former description is more appropriate).
Now as an object of the category $\C \eq\V\text{-Rep}$, $M$ has a natural structure of
$L$-module and $L$-comodule, so that we can consider its character $\chii^L_M$ and
cocharacter $\cochi^L_M$ as in \erf{Qrep2} and \erf{eq:delta,cochi}. Thinking of these
as categorical variants of the characters \erf{eq:chiv} fits perfectly with the fact
that the morphisms $\chii^L_M$ and $\cochi^L_M$, for $M$ the simple \V-mo\-du\-les, are
linearly independent and span subspaces of $\HomC(L,\one)$ and $\HomC(\one,L)$,
respectively. Likewise it fits with the observation that, at least for $\C \eq H\Mod$,
the torus partition function can be expanded as a bilinear form in the characters $\chii^L_m$
(with coefficients that are given by the Cartan matrix of \C) \cite{fuSs4}.
\\[2pt]
Of course, once we have switched to the categorical setting, the dependence on the
parameter $\tau$ (the modulus of the torus) is no longer visible. Now it follows directly
from the definition of $T_L$ that the cocharacter $\cochi^L_M$ satisfies 
  \be
  T_L \circ \cochi^L_M
  = \iL_{M^\vee_{}} \circ (\theta_M^{} \oti \id_{M^\vee_{}}) \circ \coev_{M^\vee_{}}
  \ee
with $\theta_M$ the twist isomorphism of $M$, so that acting with $T_L$ on a simple 
cocharacter just multiplies it with its twist eigenvalue.
This indicates that it is $\cochi^L_M$ that more directly corresponds to $\chiv_M(v,\tau)$,
while the character $\chii^L_M$ rather corresponds to $\chiv_M(v,-1/\tau)$. This
interpretation is corroborated by the fact that $L$ can be regarded as the dual of a
more fundamental algebraic structure (compare Example \ref{X:1}(iii)), as well as by
the following observation: Combining the relation \eqref{char-cochar} between characters
and cocharacters with the formula \erf{eq:..=SL=..} for $S_L$, we can write
$\cochi^L_m \eq (\id_L \oti (\chii^L_m \cir S_L)) \cir \omegaFm$. Now the morphisms
$\omegaF$ and $\omegaFm$ constitute a possible choice for the (right, say) evaluation
and coevaluation morphisms $\ev_L$ and $\coev_L$. With this choice, the previous equality
can be rewritten as
  \be
  \chii^L_m \circ S_L = \big(\cochi^L_m\big)^\vee_{} .
  \label{eq:in-out-dual}
  \ee
Thus in short, up to a duality the characters and cocharacters of our interest
are related by a modular S-trans\-for\-ma\-tion.
\end{rem}

%%%%%%%%%%%%%%%%%%%%%%%%%%%%%%%%%%%%%%%%%%%%%%%%%%%%%%%%%%%%%%%%%%%%%%%%

\section{Boundary states}\label{sec:bstates}

We henceforth assume that the category of boundary conditions is given by the
category \C\ of chiral sectors. Besides being a natural choice, another
attractive feature of \C\ as a module category over itself is that it is an 
exact module category, whereby it comes with a lot of additional structure, 
such as relative Serre functors \cite{fScS2}.

A boundary state assigns to a boundary condition an element of a vector space of
conformal blocks; the relevant space of conformal blocks will be described in Section
\ref{sec:bdyblocks}. As explained in Section \ref{sec:intro}, this map constitutes a
decategorification and therefore should factorize over the fusion ring, i.e.\
the Grothendieck ring of \C. As also explained there, by comparison with rational 
CFT this is achieved by realizing boundary states as characters of
the $L$-modules with the canonical $L$-action \erf{eq:rhoLm}.
In this section we discuss this issue in more detail, accounting in particular also
for the need to distinguish between incoming and outgoing bulk field insertions.
Later, in Section \ref{sec:annulus}, we will show that it also gives rise to
sensible annulus partition functions.

\subsection{Blocks for boundary states}\label{sec:bdyblocks}

We first establish that the conformal blocks for boundary states are indeed 
(canonically isomorphic to) the vector space $\EndId$. This requires some explanation,
because conformal blocks for non-semisimple conformal field theories on world sheets with
boundary have so far not been discussed systematically.

Pertinent calculations in the semisimple case have been performed in \cite{fffs3,fuRs4}. 
They made use of the fact that the complex double of
a disk is a two-sphere. Accordingly, to describe a disk with one bulk field insertion, 
without loss of generality it was possible to work with two-point blocks on the sphere
with chiral insertions given by simple objects which, in turn, can be viewed as conformal 
blocks for a three-dimensional topological field theory of Reshetikhin-Turaev type.
In contrast, for non-semisimple \C\ the Cardy bulk algebra $F$ does not decompose into 
a direct sum of products of left and right movers any longer, so that we should keep $F$
as a single insertion.  We therefore consider a different geometry: we directly
work with a disk with a single puncture in its interior, and this puncture is labeled by
the entire bulk algebra $F \iN \Z(\C)$. When doing so, then unlike in \cite{fffs3,fuRs4} 
the relevant conformal blocks are those of a topological field theory of Turaev-Viro type;
our formulation is made possible by recent insight which allows one to
consider such theories also for surfaces with boundary, and in particular for a disk. 

The vector space of conformal blocks for boundary states can thus be obtained as follows. 
First, by the results of \cite{fusV} a boundary condition for the three-dimensional 
topological field theory associated with \C\ can be characterized by a central functor
$\Z(\C) \To \W$ from the category $\Z(\C)$ of bulk Wilson lines to a fusion 
category $\W$ of boundary Wilson lines. In the Cardy case, \W\ coincides with \C\ 
and this functor is nothing but the forgetful functor $U$ from \eqref{eq:Uforget}. 
Second, by adiabatically moving the bulk insertion labeled by $F$ into the boundary, 
we obtain a boundary insertion, and this boundary insertion is labeled by the object
$U(F) \eq \Ze \iN \C$ (recall from \eqref{eq:UFL} that this has the structure of
a coend). Finally there aren't any other insertions involved. 
We conclude that the vector space in which the boundary states take their value is
  \be
  \Blout = \HomC(\one,\Ze)
  \label{eq:out-blocks}
  \ee
for the case of \emph{outgoing} boundary states, respectively
  \be
  \Blin = \HomC(\Ze,\one)
  \label{eq:in-blocks}
  \ee
for \emph{incoming} boundary states.

This conclusion holds in the first place only as long as \C\ is semisimple, which is
assumed in the setting of \cite{fusV}. It is natural to expect that, analogously as
in the case of bulk correlators \cite{fuSc22},
expressions for boundary conformal blocks remain correct in the non-semisimple case
once they are written, as we did, in a form that makes sense without assuming
semisimplicity. And indeed a detailed study of a non-semisimple generalization of
the Turaev-Viro state sum construction \cite{fScS3} confirms this expectation.

\medskip

Now recall from the Introduction that the boundary states are expected to furnish 
a ring homomorphism from the fusion ring $K_0(\C)$ of \C\ to the center of \C, 
i.e.\ to the linear natural endo-transformations $\EndId$ of the identity functor
(not to be confused with the monoidal (Drinfeld) center \ZC\ of \C). 
For compatibility, the center must be isomorphic to the spaces 
\erf{eq:out-blocks} and \erf{eq:in-blocks} of conformal blocks obtained above.
This is indeed the case: As already pointed out in formula \erf{eq:EndId=Hom} in the
Introduction, for any finite tensor category \C\ the space $\EndId$ is isomorphic as a
\koc-algebra to the morphism space $\HomC(L,\one)$ with $L \,{\equiv}\, \Ze$
(see also \Cite{Lemma\,4}{kerl5} and \Cite{Prop.\,5.2.5}{KEly}).
An isomorphism from $\HomC(L,\one)$ to $\EndId$ (as vector spaces, and even as 
algebras) is furnished by composition with the natural coaction of $L$, i.e.
  \be
  \HomC(L,\one) \ni~\alpha \,\xmapsto{~~}\,
  \big(\, (\id_x \oti (\alpha\cir\iL_x)) \cir (\coev_x \oti \id_x) \,{\big)}_{x\in\C} \,,
  \ee
and an inverse to this linear map is given by composition with the counit of $L$, i.e.\
a natural transformation $g \eq (g_x) \iN\EndId$ is mapped to the element 
$\alpha_g \iN \HomC(L,\one)$ that is determined by
  \be
  \alpha_g \circ \iL_x = \epsL \circ \iL_x \circ (\id_{x^\vee_{}} \oti g_x)
  = \ev_x \circ (\id_{x^\vee_{}} \oti g_x) \,.
  \ee

Since $L$ is self-dual, the space $\HomC(L,\one)$ is, in turn, isomorphic to $\HomC(\one,L)$.
Indeed, these two spaces are even isomorphic as algebras (with convolution product). One way
to see this is via isomorphisms to the space $\HomZ(F,F)$ of endomorphisms of the bulk object
$F$ in the center of \C; these will be presented in Section \ref{sec:annulusblocks} below. 

\begin{Example}\label{X:2}
(i)\,~If \C\ is semisimple, we have
  \be
  \HomC(L,\one) = \bigoplus_{i\in I} \HomC(x_i^\vee {\otimes}\, x_i,\one)
  \cong\, \bigoplus_{i\in I} \EndC(x_i) \,,
  \ee
so that a natural basis is given by the evaluation morphisms 
$\ev_{x_i} \iN \HomC(x_i^\vee {\otimes}\, x_i,\one)$, respectively by the identity
morphisms $\id_{x_i} \iN \EndC(x_i)$. The latter are precisely the Ishibashi states depicted
in \erf{eq:|i>>}.
\\[2pt]
(ii)\,~For the logarithmic $(p,1)$ triplet models, the space $\HomC(L,\one)$ is 
$(3p{-}1)$-dimensional. A basis of this space has been obtained in \Cite{Sect.\,3.1}{garu}
and \Cite{Sect.\,5.1}{garu2} (compare also \Cite{Sect.\,5.2}{gaTi}). There, these conformal
blocks were determined by interpreting them as bilinear forms on
$\mathcal H^\text{bulk} \times \V$ that are compatible with \V-actions and solving the
constraints, or chiral Ward identities, which implement this compatibility requirement.
\\[2pt]
(iii)\,~For $\C \eq H$\Mod\ the category of \findim\ left modules over a \findim\ factorizable
ribbon Hopf algebra $H$, $\HomC(L,\one) \,{\cong}\, \Hom_H(\ko,H)$ is the center of $H$ and 
$\HomC(\one,L) \,{\cong}\, \Hom_H(H,\ko)$ is the space of class functions, see e.g.\ \cite{coWe6}.
\end{Example}

%%%%%%%%%%%%%%%%%%%%%%%%%%%%%%%%%%%%%%%%%%%%%%%%%%%%%%%%%%%%%%%%%%%%%%%%

\subsection{Boundary states as (co)characters}
	
Recall that a  boundary state captures information about one-point functions of bulk fields
on a disk with specified boundary condition $m$. As explained in Section \ref{sec:1.1}, in
rational CFT it can be expressed as a finite sum \erf{eq:|x_a>} over simple objects $x_i$.
In the non-semisimple situation the restriction to simple objects is, however, not natural.
Category theory gives us a clear instruction what to do: consider all objects (and also account
for all morphisms). Hence we let $x$ be an arbitrary object of \C\ and use an evaluation
morphism to bend down the $x$-line in the way indicated in picture \erf{eq:openHopf2char}
-- which we repeat here for convenience (with appropriate change of labels):
     \def\locpa {1.4} 
  \be
  \begin{tikzpicture}
  \draw[very thick,color=red] (-0.08*\locpa,1.35*\locpa) arc (95:445:0.85*\locpa cm and 0.35*\locpa cm);
  \draw[very thick,color=blue] (0,0) -- (0,0.58*\locpa);
  \draw[very thick,color=blue] (0,0.74*\locpa) -- (0,1.7*\locpa);
  \node (a) at (0.48*\locpa,0.57*\locpa) {$\scs m $};
  \node (i) at (0.20,0.09) {$\scs x $};
  \end{tikzpicture}
  \raisebox{3.3em} {\hspace*{2.3em} $ \xmapsto{~~~~} $ \hspace*{1.2em} }
  \begin{tikzpicture}
  \draw[very thick,color=red] (-0.41*\locpa,0.77*\locpa) arc (222:570:0.85*\locpa cm and 0.35*\locpa cm);
  \draw[very thick,color=blue] (0,0) -- (0,0.58*\locpa);
  \draw[very thick,color=blue] (0,0.74*\locpa) .. controls (0,1.35*\locpa)
                               and (-0.8,1.35*\locpa) .. (-0.8,0) ;
  \node (a) at (0.65*\locpa,0.57*\locpa) {$\scs m $};
  \node (i) at (0.19,0.08) {$\scs x $};
  \node (ii) at (-1.07,0.08) {$\scs x^\vee $};
  \label{eq:mloop-xdown}
  \end{tikzpicture}
  \ee

Further, conceptually the summing over simple objects that is performed in the semisimple
case is a way of implementing the functoriality of conformal blocks. Its natural
non-semisimple generalization is to take the coend of the relevant functor over all 
objects $x$. Now we know from \eqref{eq:UFL} that the coend $\int^{x\in \C}\! x^\vee\oti x$
is nothing but $\Ze$, and so we indeed arrive at a vector in the space \erf{eq:in-blocks} of
conformal blocks for incoming boundary states.

Likewise, by using the coevaluation instead of the evaluation we can bend the $x$-line upwards
instead of downwards, according to
  \be
  \begin{tikzpicture}
  \draw[very thick,color=red] (-0.08*\locpa,1.35*\locpa) arc (95:445:0.85*\locpa cm and 0.35*\locpa cm);
  \draw[very thick,color=blue] (0,0) -- (0,0.58*\locpa);
  \draw[very thick,color=blue] (0,0.74*\locpa) -- (0,1.7*\locpa);
  \node (a) at (0.48*\locpa,0.57*\locpa) {$\scs m $};
  \node (i) at (0.19,0.08) {$\scs x $};
  \end{tikzpicture}
  \raisebox{3.3em} {\hspace*{2.1em} $ \xmapsto{~~~~} $ \hspace*{1.2em} }
  \begin{tikzpicture}
  \draw[very thick,color=red] (-0.06*\locpa,1.72*\locpa) arc (142:489:0.85*\locpa cm and 0.35*\locpa cm);
  \draw[very thick,color=blue] (0,2.3*\locpa) -- (0,1.35*\locpa);
  \draw[very thick,color=blue] (0,1.19*\locpa) .. controls (0,0.82*\locpa)
                               and (-0.8,0.82*\locpa) .. (-0.8,2.3*\locpa) ;
  \node (a) at (0.94*\locpa,1.05*\locpa) {$\scs m $};
  \node (i) at (0.19,2.22*\locpa) {$\scs x $};
  \node (ii) at (-1.07,2.22*\locpa) {$\scs x^\vee $};
  \node (nix) at (0,0.81) {$ $};
  \label{eq:mloop-xup}
  \end{tikzpicture}
  \ee
We then arrive again at an element in the correct space of blocks, namely the space 
\erf{eq:out-blocks} of conformal blocks for outgoing boundary states, if we take into account all
objects $x$ of \C\ as well as all relations among them. Technically, in the case at hand this is
achieved by a limit rather than colimit construction, namely taking an \emph{end} instead of 
a coend. However, owing to the existence of the non-degenerate pairing $\omegaL$, as an
object the end $\int_{x\in \C} x^\vee\oti x$ is again given (up to unique isomorphism) by $L$.

\medskip

Now comparing the picture \erf{eq:mloop-xdown} 
with the expression \erf{Qrep2} for the characters $\chii^L_m$, we can
immediately conclude that the right hand side $|m\rangle_x^\text{in}$ of \erf{eq:mloop-xdown} 
is indeed nothing but $\chii^L_m$, up to composition with a dinaturality morphism:
  \be
  \Blin \ni~ |m\rangle_x^\text{in} = \chii^L_m \circ \iL_x .
  \label{eq:mxin}
  \ee
Thus in short, we can formulate the following

\begin{pstl}\label{pstl:in}
Incoming boundary states are $L$-characters.
\end{pstl}   

For obtaining an analogous interpretation of the right hand side $|m\rangle_x^\text{out}$ of the
picture \erf{eq:mloop-xdown}, we need in addition the information that \Cite{Lemma\,2.1}{fuSc23}
the dinatural family $\jL$ for $L \eq \Ze$ as an \emph{end} can be expressed as 
 \footnote{~For defining the dinatural family of the end one could alternatively use the
 Frobenius form instead of the Hopf pairing. By the universality of the end the two
 families must be isomorphic, and indeed they differ just by composition with $S_L$.}
  \be
  \HomC(L,x^\vee{\otimes}\,x) \ni~ 
  \jL_x = \big( \id_{x^\vee_{}\otimes x} \oti [ \omegaL \cir (\iL_x \oti \id_L)] \big)
  \circ (\coev_{x^\vee_{}\otimes x} \oti \id_L)
  \ee
in terms of the dinatural family $\iL$ for $L$ as a coend
(here we again suppress the pivotal structure of \C).
Combining this relation and the result \erf{eq:mxin} with the expression \erf{eq:cochi^L_m}
for cocharacters and with the formula \erf{char-cochar} for the relation between characters 
and cocharacters, it follows that, similarly as in \erf{eq:mxin}, we have
  \be
  \Blout \ni~ |m\rangle_x^\text{out} = \jL_x \circ \cochi^L_m .
  \label{eq:mxout}
  \ee
We can again express this result as a

\begin{pstl}\label{pstl:out}
Outgoing boundary states are $L$-cocharacters.
\end{pstl}   

\medskip

That, as seen in Lemma \ref{lem:ringhomom}, the characters and cocharacters neatly provide us 
with ring homomorphisms from the fusion ring $K_0(\C)$ to the center of \C, convincingly
supports these conclusions. 

\begin{rem}
If \C\ is non-semisimple, then the characters and cocharacters, and thus the outgoing 
and incoming boundary states, only span proper subspaces of the respective spaces $\Blout$
and $\Blin$ of conformal blocks. For the torus partition function The same pattern is known
\Cite{Rem.\,4(ii)}{fuSs4} to be present if \C\ is the \rep\ category of a non-semisimple
Hopf algebra, and is expected to arise analogously if \C\ is any non-semisimple finite CFT.
\end{rem}

\begin{rem}
For a non-semisimple algebra, characters do not fully specify a 
\rep, since they split under extensions (see formula \erf{eq:split}). Hence in the
non-semisimple case, boundary conditions cannot be fully distinguished by their boundary 
states, and thus cannot be classified completely with the help of boundary states.
(Likewise, the torus partition function does not fully specify the bulk state space.)
\end{rem}

\begin{Example}\label{X:3}
(i)\,~If \C\ is semisimple, then the characters $\chi^L_m$ span all of $\HomC(L,\one)$.
Also note that the action \erf{eq:rhoLm} of $L$ involves a double braiding. For 
semisimple \C, this gives rise to the appearance of entries of the (Hopf link)
S-matrix $\soo$ and thereby reproduces the familiar formula \erf{|x_a>} for the Cardy
case boundary states of rational CFTs.
\\[2pt]
(ii)\,~For the logarithmic $(p,1)$ triplet models, there are $2p$ simple objects and thus
$2p$ Cardy boundary states. This agrees with the findings in
\Cite{Sect.\,3.3}{garu} and \Cite{Sect.\,5.2}{garu2}. Note, however, that the approach of
\cite{garu,garu2} for determining boundary states is quite different from ours. Namely,
it consists of making an ansatz for the boundary states as linear combinations of 
Ishibashi states which form a basis of the relevant space of conformal blocks (similar to the
states $|i\rrangle$ that are given by \erf{eq:|i>>} for the rational case) and then imposing
the so-called Cardy condition for the resulting annulus amplitudes. The solution to this
constraint, which is found with the help of explicit expressions for the characters of the
triplet vertex operator algebra and their modular transformations, is not unique.
In \Cite{Sect.\,5.3}{gaTi} additional boundary states were suggested, yielding a total
of $3p\,{-}\,1$ boundary states; including these leads to annulus coefficients which are
integral but, unlike those for the $2p$ boundary states obtained in \cite{garu2} (see Example 
\ref{X:4} below) can be negative.
\\[2pt]
(iii)\,~For $\C \eq H$\Mod\ the category of \findim\ left modules over a \findim\ factorizable
ribbon Hopf algebra $H$, the subspace of $\HomC(L,\one)$, i.e.\ of the center of $H$, that is 
spanned by the characters $\chii^L_m$
is the Reynolds ideal of $H$, i.e.\ the intersection of the socle of $H$ with the center,
while the subspace of $\HomC(\one,L)$, i.e.\ of the class functions, is the character algebra 
of $H$, i.e.\ the span of all $H$-characters; see e.g.\ \cite{coWe6,loren3}.
Further, the $L$-characters $\chii^L_m$ are related to the $H$-characters by the Drinfeld map
\Cite{Lemma\,6}{fuSs4}, while with the help of the explicit form of the dinatural morphisms
$\iL_m$ (see \Cite{Prop.\,A.6}{fuSs3}) one shows that the $L$-cocharacters just coincide with
the $H$-characters as morphisms in $H$\Mod.
Also, the fact that the Hopf algebra $H$ is a symmetric algebra allows one to understand
aspects of the complement of the character algebra in the space of class functions 
in terms of linear functions on the endomorphism rings of projective $H$-modules \cite{arik3}.
\end{Example}

%%%%%%%%%%%%%%%%%%%%%%%%%%%%%%%%%%%%%%%%%%%%%%%%%%%%%%%%%%%%%%%%%%%%%%%%

\section{Annulus amplitudes}\label{sec:annulus}

In this final section we compute the annulus amplitudes that result from our proposal for the 
boundary states for the Cardy case of a finite CFT. Concretely, we obtain the correlator for 
an annulus with boundary conditions $m$ and $n$ and without field insertions by sewing the
one-point correlators for bulk fields on two disks having boundary conditions $m$ and $n$,
respectively, i.e.\ of the boundary states for $m$ and $n$. As in Section
\ref{sec:bstates} we start by describing the relevant spaces of conformal blocks.

\subsection{Blocks for annulus amplitudes}\label{sec:annulusblocks} 

As pointed out above, for non-semisimple conformal field theories, conformal blocks for world 
sheets with boundary have so far not been constructed systematically. But for partition functions,
i.e.\ correlators without field insertions, it is quite irrelevant whether the bulk object
is of factorized form, as in the semisimple case, or not. We can therefore transfer the 
insight \cite{fffs3,fuRs4} that the conformal blocks for a correlator are those for the
complex double of the relevant conformal surface from the semisimple to the general case.
Now the complex double of an annulus is a torus. Moreover, the space 
of zero-point conformal blocks on a torus is known \cite{lyub11,fuSc22} to be 
isomorphic to the morphism spaces $\HomC(\one,L)$ and $\HomC(L,\one)$, i.e.\ to the spaces
\erf{eq:out-blocks} and \erf{eq:in-blocks} for outgoing and incoming boundary states which, in 
turn, are (see \erf{eq:EndId=Hom})
isomorphic as algebras to the natural endo-transformations of the identity functor:
  \be
  \BlA \cong \Blout  \cong \Blin \cong \EndID \,.
  \label{eq:BlA-cong-1}
  \ee
To be precise, when specifying an annulus partition function we must also tell in which `channel'
it is considered -- closed-string channel or open-string channel -- or, in the more precise terminology
of the Lego-Teichm\"uller game \cite{baKir,fuSc22}, with which auxiliary marking the annulus
is endowed. In accordance with the observation about the connection with vertex algebra characters
in Remark \ref{rem:chiv}, the open-string channel should correspond to the space $\HomC(\one,L)$.
It is this space that we mean in the sequel when referring to annulus blocks, i.e.\ we set
  \be 
  \BlA = \HomC(\one,L) = \Blout \,.
  \label{eq:BlA=Blout}
  \ee

For the discussion of sewing, yet another description of this space will be convenient,
namely as a morphism space in the Drinfeld center:
  \be
  \BlA \cong \EndZ(F) \,.
  \label{eq:BlA-cong-2}
  \ee

To arrive at this description we make use of the fact (see Appendix \ref{app:centralmonad}) that
for a modular tensor category \C\ the Drinfeld center \ZC\ is equivalent to the category of
$L$-modules in \C. Recall from Section \ref{sec:bulkalgebra} that $L \eq \Ze \eq U(F)$ has 
a structure of a Frobenius algebra in \C, and that the induction functor $I$ from \C\ to \ZC\
and forgetful functor $U\colon \ZC \To \C$ form a bi-adjoint pair. 
The adjunction morphisms are given by the elementary \rep-theoretic formulas \erf{eq:Ula=I} 
and \erf{eq:Ura=I}, respectively, specialized to the case at hand, in which $I(\one) \eq F$. 
For completeness, let us write down the linear isomorphisms 
  \be
  \vphi: ~ \HomC(\one,L) = \HomC(\one,U(F)) \,\xleftrightarrow{~\cong~}\,
  \HomZ(I(\one),F) = \HomZ(F,F) ~: \vpsi
  \label{eq:out-blocks-2}
  \ee
and
  \be
  \tphi: ~ \HomC(L,\one) = \HomC(U(F),\one) \,\xleftrightarrow{~\cong~}\, 
  \HomZ(F,I(\one)) = \HomZ(F,F) ~: \tpsi
  \label{eq:in-blocks-2}
  \ee
which express that $I$ is, respectively, a left and right adjoint of $U$, explicitly.
They are given by
  \be
  \vphi(\tild\alpha) = \mulF \circ (\id_F \oti \tild\alpha) \qquand
  \vpsi(g) = g \circ \etaF \,,
  \label{eq:1,L<->F,F-2}
  \ee
and by
  \be
  \tphi(\alpha)
  = \big( \id_F \oti (\alpha \cir\mulF)\big) \circ \big( (\DeltaF\cir\etaF) \oti \id_F \big)
  = (\id_F \oti \alpha) \circ \DeltaF
  \label{eq:L,1<->F,F-2a}
  \ee
and
  \be
  \tpsi(g) := \epsF \circ g \,,
  \label{eq:L,1<->F,F-2b}
  \ee
respectively.
(The second expression in \eqref{eq:L,1<->F,F-2a} follows from the first by using
the Frobenius and unit axioms.)

Note that in the formulas above the endomorphisms of $F$ are described as morphisms in \C.
That they are indeed even morphisms in \ZC\ may be verified by direct calculation, but
can be seen more conceptually by noting that they can also be expressed in terms of the
central monad on \C\ (compare the formulas \erf{eq:1,L->F,F} and \erf{eq:F,F->1,L}).
Further, the linear maps \erf{eq:out-blocks-2} are in fact algebra isomorphisms, with the 
multiplication on $\HomZ(F,F)$ being composition and the one on $\HomC(\one,L)$ given by 
the convolution product \eqref{tildalpha.tildbeta} \Cite{Thm.\,3.9}{shimi9}.

\begin{rem}
The compatibility between bulk and boundary theories is of widespread interest
in various settings of quantum field theory.
The isomorphisms \erf{eq:BlA-cong-1} and \erf{eq:BlA-cong-2} can be viewed as one aspect
of such a compatibility: $\HomC(L,\one)$ and $\HomC(\one,L)$ have a natural interpretation
as spaces for boundary states, while $\EndZ(F)$ appears naturally when studying bulk fields. 
\end{rem}

%%%%%%%%%%%%%%%%%%%%%%%%%%%%%%%%%%%%%%%%%%%%%%%%%%%%%%%%%%%%%%%%%%%%%%%%

\subsection{Sewing of boundary blocks}

For understanding the relation between boundary states and annulus 
amplitudes we need to know how factorization relates conformal blocks for disks to 
annulus blocks, i.e.\ to find four sewing maps
  \be
  \Blinout \oticx \Blinout \xrightarrow{~~~} \BlA
  \label{eq:out-in->ann}
  \ee
from the tensor product of the conformal blocks for incoming and/or outgoing boundary states 
to the space of annulus blocks. Clearly, these maps are not bijections; this reflects the
fact that no summation over the bulk insertions on the boundary disks is implied.

The following two ingredients are crucial for being able to understand these sewing maps:
First, with the help of the results of Section \ref{sec:annulusblocks}
we can rewrite the maps \erf{eq:out-in->ann} as the composition of known isomorphisms with 
a sewing map $s$ that only involves morphism spaces in \ZC\ or, in other words, 
conformal blocks for the bulk theory. Specifically,
  \be
  \begin{array}{ll}
  \Blout \oticx \Blin = \HomC(\one,L) \oticx \HomC(L,\one) \!\!\!& 
  \xrightarrow[~\,\vphi\otimes\tphi~]{\cong}\, \EndZ(F) \oticx \EndZ(F)
  \Nxl2 &
  \xrightarrow{~~~s~~~}\, \EndZ(F) \,\xrightarrow[~~\vpsi~~]{\cong}\, \BlA 
  \eear
  \label{eq:1LxL1->FFxFF->FF->L1}
  \ee
for the case of a sewing of incoming and outgoing boundary blocks, and analogously for 
other combinations.

Second, the sewing of morphism spaces of $\ZC \,{\simeq}\, \CbC$ is fully understood
\cite{lyub11,fuSc22}; it is realized by the structural maps of appropriate coends in \ZC. 
Moreover, for the morphism spaces in question these structural maps just amount to 
composition of linear maps (see e.g. \Cite{Cor.\,2.3}{fuSc23}).
Thus the sewing map $s$ in \eqref{eq:1LxL1->FFxFF->FF->L1} takes the simple form
  \be
  s:\quad \EndZ(F) \oticx \EndZ(F) \ni~
  f_1 \oti f_2 \,\longmapsto\, f_2 \cir f_1 ~\in \EndZ(F) \,.
  \label{eq:sewing}
  \ee

Let us write down the resulting sewing maps \erf{eq:out-in->ann} explicitly.
We start with the sewing of outgoing to outgoing boundary blocks, which will confirm the
identification \erf{eq:BlA=Blout} of the (open-string channel) annulus amplitude. We have
  \be
  \bearll
  \tild\alpha \oti \tild\beta \!\!&
  \xmapsto{~\vphi\otimes\vphi~} 
  (\mulF \cir (\id_F \oti \tild\alpha)) \otimes (\mulF \cir (\id_F \oti \tild\beta))
  \Nxl2 &
  \xmapsto{~~~s\,~~}\, \mulF \cir (\mulF \oti \tild\beta) \cir (\id_F \oti \tild\alpha)
  \Nxl2 & \hsp{1.62}
  =\, \mulF \cir \big( \id_F \oti [ \mulF \cir (\tild\alpha \oti \tild\beta) ] \big)
  \,\equiv\, \vphi(\tild\alpha \cvp \tild\beta)
  \,\xmapsto{~~\vpsi~~}\, \tild\alpha \cvp \tild\beta
  \eear
  \label{eq:sew=convolute}
  \ee
for any $\tild\alpha,\tild\beta \iN \HomC(\one,L)$, where the equality holds by associativity.
In short, $\Blout$-$\Blout$-sewing amounts to convolution in $\HomC(\one,L)$.

\medskip

The sewing of outgoing to incoming boundary blocks (or vice versa) amounts to a convolution 
as well provided, however, that we also include a modular S-transformation, as might be
anticipated from the observation \erf{eq:in-out-dual}. Indeed, directly we have
  \be
  \bearll
  \tild\alpha \oti \beta \!\!&
  \xmapsto{~\vphi\otimes\tphi~} 
  (\mulF \cir (\id_F \oti \tild\alpha)) \otimes ((\id_F \oti \beta) \cir \Delta_F)
  \Nxl2 &
  \xmapsto{~~\,s~~~} (\id_F \oti \beta) \circ \Delta_F \circ 
  \mulF \circ (\id_F \oti \tild\alpha)
  = [ \id_F \oti (\beta \cir \mulF) ] \cir ( \Delta_F \oti \tild\alpha )
  \eear
  \ee
for any $\tild\alpha \iN \HomC(\one,L)$ and $\beta \iN \HomC(L.\one)$, where the equality uses
the Frobenius axiom. By direct rewriting, using the first of the expressions in \erf{eq:DeltaF=A10}
for the Frobenius coproduct $\DeltaF$ and formula \erf{eq:omegaFm} for $\omegaFm$, we then have
  \be
  \bearll
  s \circ (\vphi\oti\tphi)\, \big(\tild\alpha \oti (\beta \cir S_L) \big) \!\!&
  = \big( \id_L \otimes (\beta \cir S_L \cir \mulF) \big) \circ (\DeltaF \oti \tild\alpha)
  \Nxl2 &
  = \big( \id_L \otimes (\beta \cir S_L \cir \mulF) \big) \circ (\id_L \oti \mulF \oti \tild\alpha)
  \circ (\omegaFm \oti \id_L) \,.
  \eear
  \ee
Inserting the explicit form \erf{eq:S=omegaF.omegaLm} of the endomorphism $S_L$ 
and using associativity of $\mulF$ together
with the invariance of the Frobenius form this, in turn, implies that
  \be
  \bearll
  s \circ (\vphi\oti\tphi)\, \big(\tild\alpha \oti (\beta \cir S_L) \big) \!\!&
  = (\mulF \oti \beta) \circ (\mulF \oti \omegaLm) \circ (\id_L \oti \tild\alpha)
  \Nxl2 &
  = \mulF \cir \big( \id_F \oti [ \mulF \cir (\tild\alpha \oti \tild\beta) ] \big)
  \,\equiv \vphi(\tild\alpha \cvp \tild\beta) \,,
  \eear
  \label{eq:out-sew-in.S}
  \ee
where $\tild\beta \iN \HomC(\one,L)$ is defined analogously as in the relation \erf{char-cochar}
between characters and cocharacters, i.e.
  \be
  \tild\beta := (\id_L \oti \beta) \circ \omegaLm .
  \label{eq:tildebeta}
  \ee
Thus also the sewing of an outgoing and an incoming boundary block can be described as a
convolution in the space $\HomC(\one,L)$, provided that before sewing the incoming boundary 
block is precomposed with an S-transformation. For performing the convolution, the incoming
boundary block is transformed to an outgoing one via the side inverse of the Hopf pairing.

Finally, the sewing of two incoming boundary blocks amounts to flipped $\DeltaF$-convolution:
we have
  \be
  \bearll
  \alpha \oti \beta \!\!& \xmapsto{~\tphi\otimes\tphi~} 
  ((\id_F \oti \alpha) \cir \Delta_F) \otimes ((\id_F \oti \beta) \cir \Delta_F)
  \Nxl2 &
  \xmapsto{~~~s\,~~} (\id_F \oti \beta) \cir (\Delta_F \oti \alpha) \cir \Delta_F
  = \big( \id_F \oti [ (\beta \oti \alpha) \cir \Delta_F] \big) \cir \Delta_F
  \equiv \tphi(\beta \cvpF \alpha) 
  \Nxl1 &
  \xmapsto{~~~\tpsi~~} \beta \cvpF \alpha
  \eear
  \label{eq:sew=convolute2}
  \ee
for any $\alpha,\beta \iN \HomC(L,\one)$, where the equality holds by coassociativity.
In view of the result \erf{eq:out-sew-in.S} above, it is natural to consider alternatively
the sewing after precomposing both of the incoming boundary blocks with $S_L$,
  \be
  (\alpha \cir S_L) \oti (\beta  \cir S_L) \,\xmapsto{~~~~}\,
  (\beta\cir S_L) \cvpF (\alpha\cir S_L) \,.
  \label{eq:a.SL-sew-b.SL}
  \ee
A straightforward calculation shows
that the sewn expression in \erf{eq:a.SL-sew-b.SL} can be rewritten as 
  \be
  (\beta \cir S_L) \cvpF (\alpha \cir S_L)
  = \omegaF \circ \big( \id_L \oti (\tild\alpha \cvp \tild\beta) \big)
  \equiv \omegaL \circ \big( \id_L \oti [ S_L \cir (\tild\alpha \cvp \tild\beta)] \big) 
  \label{eq:a.SL-sew-b.SL-2}
  \ee
with $\tild\beta$ as in \erf{eq:tildebeta} and analogously for $\tild\alpha$.
Thus also in this case we get an expression that is a convolution product, up to
suitable insertions of $S_L$-automorphisms
and using $\omegaL$ and its side inverse to relate incoming and outgoing boundary blocks.

\medskip

To summarize, the \emph{sewing of two boundary blocks amounts to a suitable convolution}.

\begin{rem}\label{rem:S}
The occurrence of an additional S-transformation when switching between outgoing and
incoming boundary blocks should not come as a surprise once one remembers that
according to formula \eqref{eq:in-out-dual}, characters and cocharacters are related to 
each other by a duality morphism combined with an S-transformation.
More significantly, this behavior can be appreciated through a comparison with the 
description of sewing that is natural in the context of the TFT construction of the
correlators of rational CFTs. In that context, sewing combines an outgoing and an 
incoming block, and it automatically involves a modular S-transformation;
see e.g.\ Sections 5.2 and 5.3 of \cite{fjfrs} or Section 2.2 of \cite{fjfs} for details.
For the same type of sewing, a shadow of that S-transformation appears in the present 
approach, namely as the morphism $S_L$ in \erf{eq:out-sew-in.S}.
\end{rem}

%%%%%%%%%%%%%%%%%%%%%%%%%%%%%%%%%%%%%%%%%%%%%%%%%%%%%%%%%%%%%%%%%%%%%%%%

\subsection{Annulus amplitudes via sewing of boundary states}

Knowing the sewing maps on the respective spaces of boundary blocks, we can proceed to
the sewing of the boundary states proposed above as specific elements of those spaces.

Consider first the sewing of two outgoing boundary states corresponding to boundary
conditions $m$ and $n$. According to \erf{eq:mxout} these are given by the cocharacters
$\cochi^L_m$ and $\cochi^L_n$. The sewing relation \erf{eq:sew=convolute}, combined
with \erf{eq:mn=m*n} tells us that the (open-string channel) annulus amplitude with
boundary conditions $m,n\iN\C$ is 
  \be
  \BlA \ni~ \mathrm A_{mn} = \cochi^L_m \cvp \cochi^L_n = \cochi^L_{m \otimes n} \,.
  \ee

Now the mapping $[m] \,{\mapsto}\, \cochi^L_m$ is an injective homomorphism of monoids
from the Grothendieck ring $K_0(\C)$ to $\HomC(\one,L)$ \Cite{Cor.\,4.2}{shimi9}, so that
the annulus amplitude can be written as
  \be
  \mathrm A_{mn} = \cochi^L_{m \otimes n} = \sum_{l\in I} N_{mn}^{\phantom{m.}l}\, \cochi^L_l
  \label{eq:Amn=Nmnl.cochiLl}
  \ee
with $N_{mn}^{\phantom{m.}l} \iN \zet_{\ge0}$ the structure constants of $K_0(\C)$.
Thus $\mathrm A_{mn}$ can be expanded as a linear combination of cocharacters, with
coefficients given by the structure constants of the Grothendieck ring. In particular,
$\mathrm A_{mn}$ lies completely in the subspace of the space $\HomC(\one,L)$ of conformal
blocks that is spanned by the cocharacters -- recall that this is a proper subspace of
non-zero codimension unless \C\ is semisimple. In short, \emph{the annulus coefficients
are fusion coefficients}.

Since according to Remark \ref{rem:chiv} the cocharacters correspond to the genus-1 one-point
functions of vertex algebra \rep s, and since fusion coefficients are by definition non-negative 
integers, we see in particular that the annulus amplitude $\mathrm A_{mn}$ naturally
admits an interpretation as a \emph{partition function} that counts (open-string) states,
precisely like in the semisimple case. This constitutes a rather non-trivial check of the
interpretation of boundary states as (co)characters. 

\begin{Example}\label{X:4}
(i)\,~For semisimple \C, the formula \erf{eq:Amn=Nmnl.cochiLl} reproduces the familiar result 
\cite{card9} for the Cardy case annulus amplitude in the open-string channel. Note that the
semisimple case is too degenerate -- in that case the cocharacters span all of $\HomC(\one,L)$
-- to illustrate the non-trivial feature of \erf{eq:Amn=Nmnl.cochiLl} that the
`pseudo-cocharacters' do not contribute to the amplitude.
\\[2pt]
(ii)\,~Annulus amplitudes for the logarithmic $(p,1)$ triplet models have been computed in
\Cite{Sect.\,3.2}{garu} and \Cite{Sect.\,5.2}{garu2}. Specifically, once cocharacters are
identified with vertex algebra one-point functions, \erf{eq:Amn=Nmnl.cochiLl} reproduces formula 
(5.10) of \cite{garu2}, which in the context considered there arises as an assumption
that precedes the determination of the boundary states.
\end{Example}

The other types of sewings considered in the previous subsection amount to the following
statements for annulus amplitudes. For the sewing of an incoming and an outgoing boundary 
state, the formula \erf{eq:out-sew-in.S} yields
  \be
  \cochi^L_m \,\otimes\, \chii^L_n ~\xmapsto{~~{\rm sew}~~}~ 
  \cochi^L_m \,\cvp\, \Om^{-1}( \chii^L_n \cir S_L^{-1} ) \,,
  \label{cochi-chii-sewing}
  \ee
while according to \erf{eq:a.SL-sew-b.SL-2}, sewing two incoming boundary states is given by
  \be
  \chii^L_m \,\otimes\, \chii^L_n ~\xmapsto{~~{\rm sew}~~}~
  S_L^{} \circ \big( \Om^{-1}( \chii^L_m \cir S_L^{-1} ) \cvp \Om^{-1}( \chii^L_n \cir S_L^{-1} ) 
  \big) \,.
  \ee
Here $\Om\colon \HomC(\one,L) \To \HomC(L,\one)$ is the composition with the Hopf pairing, as
defined in \erf{eq:defOm}, which exchanges characters and cocharacters.

The presence of the modular S-transformation $S_L \iN \EndC(L)$ in \erf{cochi-chii-sewing}
has a natural interpretation in terms of the TFT construction of the correlators, see
Remark \ref{rem:S} above. Another way to understand it is obtained by realizing that the 
annulus amplitude described by \erf{cochi-chii-sewing} is in the closed-string channel -- the
formula thus tells us that the annulus amplitudes in the open- and closed-string channels are 
related by an S-transformation.

%%%%%%%%%%%%%%%%%%%%%%%%%%%%%%%%%%%%%%%%%%%%%%%%%%%%%%%%%%%%%%%%%%%%%%%%

\subsection{Boundary fields}

Finally we note that the interpretation of the annulus amplitudes as partition functions
implies constraints on the ``objects of boundary fields''. More specifically, the boundary
fields that change a boundary condition $m$ to the boundary condition $n$ must be compatible
with the annulus amplitude $\mathrm A_{m^\vee n}$ as given by \erf{eq:Amn=Nmnl.cochiLl}. There
is a priori no guarantee that those constraints can be satisfied at all. But inspection shows
that this is indeed possible, namely by taking these fields to be given by the corresponding
\emph{internal Hom} for \C\ as a (right) \C-module category, i.e.\ by
  \be
  B_{mn} := \underline{\Hom}(m,n) \cong m^\vee \otimes n ~\in \C \,.
  \ee
This natural prescription generalizes the finding in the semisimple case; that it is compatible
with the annulus amplitudes gives additional support to our ansatz for the boundary states.
  
The collection of objects $B_{mn}$ is not only naturally associated with the structure of \C\
as a module category over itself, but also comes with further structure: It can be endowed with 
natural algebroid and coalgebroid structures that fit together as a \emph{Frobenius algebroid}
in \C. The non-zero components of the product of the algebroid and of the coproduct of the
coalgebroid are
  \be
  \bearll
  \mul_{l,m,n} := \id_{l^\vee} \otimes \tev_m \otimes \id_n : & B_{lm}\oti B_{mn} \to B_{ln} 
  \qquad {\rm and}
  \nxl2
  \Delta_{l,m,n} := \id_{l^\vee} \oti \coev_m \oti \id_n :~~ & B_{ln} \to B_{lm} \oti B_{mn} \,,
  \label{eq:(co)productoid}
  \eear
  \ee
while the components of the unit and counit are $\eta_m \eq \tcoev_m \colon \one \To B_{mm}$  
and $\eps_m \eq \ev_m \colon B_{mm} \To \one$, respectively. The Frobenius compatibility
condition satisfied by the morphisms \erf{eq:(co)productoid} reads
  \be
  \bearll
  (\id_{m^\vee\otimes k} \oti \mul_{k,l,n}) \circ (\Delta_{m,k,l} \oti \id_{l^\vee\otimes n})
  \!\!& = \Delta_{m,k,n} \circ \mul_{m,l,n}
  \Nxl2 &
  = (\mul_{m,k,l} \oti \id_{l^\vee\otimes n}) \circ (\id_{m^\vee\otimes k} \oti \Delta_{k,l,n}) \,.
  \eear
  \ee
Associativity of the algebroid is in fact a consequence of the realization by
internal Homs, even beyond the setting of finite tensor categories \Cite{Sect.\,3.3}{garW}.

In particular, the objects $B_{mm}$ which describe boundary fields that do not change the
boundary condition admit a structure of symmetric Frobenius algebra, with product
$\mul_{m,m,m}$, unit $\eta_m$, coproduct $\Delta_{m,m,m}$ and counit
$\eps_m$, which is in addition special iff $\dim_\C(m) \,{\ne}\, 0$.
 
\medskip

The algebroid structure of the objects $B_{mn}$ should realize the operator product of 
boundary fields, in much the same way as the multiplication $\mulF$ on the bulk 
object $F$ realizes the operator product expansion of bulk fields (and makes the 
associativity property of that operator product precise).
We expect that the description of boundary fields as internal Homs will generalize to
the non-Cardy case, in which the boundary conditions should still be the objects of some
suitable module category $\mathcal M$ over \C.

%%%%%%%%%%%%%%%%%%%%%%%%%%%%%%%%%%%%%%%%%%%%%%%%%%%%%%%%%%%%%%%%%%%%%%%%
\vskip 4em

\noindent
{\sc Acknowledgements:}\\[.3em]
We are grateful to Ingo Runkel and Azat Gainutdinov
for discussions and helpful comments on the manuscript.
TG thanks JF and the Karlstad Physics Department for their hospitality during his visit
at the start of this project; JF thanks TG and Thomas Creutzig for their hospitality 
at UofA while the paper was being completed.
\\
JF is supported by VR under project no.\ 621-2013-4207. TG is supported in part by NSERC.
CS is partially supported by the Collaborative Research Centre 676 ``Particles,
Strings and the Early Universe - the Structure of Matter and Space-Time'' and by the 
RTG 1670 ``Mathematics inspired by String theory and Quantum Field Theory''.

%%%%%%%%%%%%%%%%%%%%%%%%%%%%%%%%%%%%%%%%%%%%%%%%%%%%%%%%%%%%%%%%%%%%%%%%
\newpage 
\appendix

\section{Appendix}

\subsection{A canonical functor from the enveloping category to the Drinfeld center}\label{more:GC}

Here we provide further details about the canonical functor $\GC\colon \Cb \boti \C \to \Z(\C)$ 
that is introduced in \eqref{eq:CbC->ZC} for braided finite tensor categories \C.
Specifying a monoidal structure on \GC\ amounts to giving coherent isomorphisms 
  \be 
  \bearll
  \varphi_{\Ol u,v,\Ol x,y}^{} : &
  \Ol u \oti v \oti \Ol x \oti y = \GC(\Ol u \boti v) \otiz \GC(\Ol x \boti y)
  \Nxl2&
  \hsp{6,5} \xrightarrow{~~~~}~ \GC\big( (\Ol u\boti v) \oticc (\Ol x \boti y) \big)
  = \Ol u \oti \Ol x \oti v \oti y 
  \eear
  \label{eq:GCmonoidal}
  \ee
that obey the relevant hexagon and triangle identities. 
It is easily checked that any odd power of the braiding (in \C)
of the second and third tensor factor, i.e.
  \be
  \varphi_{\Ol u,v,\Ol x,y}^{(m)} := \id_u \oti (\cb^{2m+1}_{})_{v,\Ol x}^{} \oti \id_y
  \label{eq:varphi(m)}
  \ee
for any $m \iN \zet$ satisfies these requirements.

In order that the functor is even \emph{braided} monoidal, the morphisms $\varphi$
must in addition make the diagrams
  \be
  \begin{tikzcd}[column sep=13ex]
  \Ol u \oti v \oti \Ol x \oti y \ar{d}[swap]{\cz_{\Ol u\otimes v,\Ol x\otimes y}} 
  \ar{r}{\varphi_{\Ol u,v,\Ol x,y}^{}} 
  & \Ol u \oti \Ol x \oti v \oti y \ar{d}{\GC(\ccc_{\Ol u\otimes \Ol x,v\otimes y}^{})}
  \\
  \Ol x \oti y \oti \Ol u \oti v \ar{r}{\varphi_{\Ol x,y,\Ol u,v}^{}}
  & \Ol x \oti \Ol u \oti y \oti v
  \end{tikzcd}
  \label{tikz:zc-cc}
  \ee
commute.
Here $\Ccc$ is the braiding on the enveloping category \CbC\ that is induced from the one of 
\C, i.e.
  \be
  \ccc_{\Ol u \otimes \Ol x, v \otimes y}^{} = \cb_{\Ol u,\Ol x}^{-1} \otimes \cb_{v,y}^{} \,,
  \ee
while $\cz$ is the braiding in $\Z(\C)$ which just coincides with the half-braiding, i.e.\
is given by \eqref{eq:cB} with $c \eq \Ol x \oti y$ ($\cz_{\Ol u\otimes v,\Ol x\otimes y}$
must not be confused with the braiding of $\Ol u\oti v$ with $\Ol x\oti y$ in \C). Pictorially,
  \be
  \begin{tikzpicture}
  \braid[ line width=\piclinesize, height=2.1cm, style strands={1,2,3,4}{blue},
  ] (braid) at (-3,0) s_1-s_3^{-1};
  \node[at=(braid-rev-1-e),below=\piclabetsep] {$\Ol u$}; 
  \node[at=(braid-rev-2-e),below=\piclabetsep] {$\Ol x$}; 
  \node[at=(braid-rev-3-e),below=\piclabetsep] {$v$}; 
  \node[at=(braid-rev-4-e),below=\piclabetsep] {$y$}; 
  \node (t1) at (-5.1,-1.4) {$\ccc_{\Ol u \otimes \Ol x, v \otimes y}^{} ~=$};
  \braid[ line width=\piclinesize, height=0.7cm, style strands={1,2,3,4}{blue},
  ] (braid) at (4.9,0) s_2 s_1-s_3^{-1} s_2^{-1};
  \node[at=(braid-rev-1-e),below=\piclabetsep] {$\Ol u$}; 
  \node[at=(braid-rev-2-e),below=\piclabetsep] {$v$}; 
  \node[at=(braid-rev-3-e),below=\piclabetsep] {$\Ol x$}; 
  \node[at=(braid-rev-4-e),below=\piclabetsep] {$y$}; 
  \node (t2) at (3.1,-1.4) {$\cz_{\Ol u\otimes v,\Ol x\otimes y} ~=$};
  \end{tikzpicture}
  \ee
One finds that this requirement is solved uniquely by setting $m \eq 0$ in \eqref{eq:varphi(m)}, 
i.e.\ we have

\begin{lem}\label{lem:braidedstructureonGC}
The braided monoidal structure \erf{eq:GCmonoidal} on the functor \GC\ is given by
  \be
  \varphi_{\Ol u,v,\Ol x,y}^{} = \id_{\Ol u} \oti \cb_{v,\Ol x}^{} \oti \id_y 
  \ee
for $u,v,x,y \iN \C$,
\end{lem}

\begin{proof}
That the diagram \eqref{tikz:zc-cc} commutes if $m \eq 0$ is seen diagrammatically as follows:
  \be
  \begin{tikzpicture}
  \braid[ line width=\piclinesize, height=1.4cm, style strands={1,2,3,4}{blue},
  ] (braid) at (-2,0) s_1-s_3^{-1} | s_2^{-1};
  \node[at=(braid-rev-1-e),below=\piclabetsep] {$\Ol u$}; 
  \node[at=(braid-rev-2-e),below=\piclabetsep] {$v$}; 
  \node[at=(braid-rev-3-e),below=\piclabetsep] {$\Ol x$}; 
  \node[at=(braid-rev-4-e),below=\piclabetsep] {$y$}; 
  \node (eq) at (-2.87,-2.4) {$\ssg \varphi_{\Ol u,v,\Ol x,y}^{} $};
  \node (eq) at (2.75,-1.8) {$ = $};
  \braid[ line width=\piclinesize, height=0.7cm, style strands={1,2,3,4}{blue},
  ] (braid) at (4.4,0) | s_2^{-1} s_2 s_1-s_3^{-1} s_2^{-1};
  \node[at=(braid-rev-1-e),below=\piclabetsep] {$\Ol u$}; 
  \node[at=(braid-rev-2-e),below=\piclabetsep] {$v$}; 
  \node[at=(braid-rev-3-e),below=\piclabetsep] {$\Ol x$}; 
  \node[at=(braid-rev-4-e),below=\piclabetsep] {$y$}; 
  \node (eq) at (8.27,-0.63) {$\ssg \varphi_{\Ol u,v,\Ol x,y}^{} $};
  \end{tikzpicture}
  \ee
This picture also makes it clear that the requirement cannot be fulfilled with any
other odd power of the braiding.
\end{proof}

%%%%%%%%%%%%%%%%%%%%%%%%%%%%%%%%%%%%%%%%%%%%%%%%%%%%%%%%%%%%%%%%%%%%%%%%

\subsection{Adjoints of the forgetful functor}\label{more:U-I}

Given an (associative, unital) algebra in a monoidal category \C, we denote
by $\M \,{\equiv}\, A\Mod$ the category of left $A$-modules in \C. Further, write
  \be
  \begin{array}{rl}
  U \colon \quad \M &\!\!\! \xrightarrow{~~~} \C 
  \nxl1
  m &\!\!\! \longmapsto \dot m
  \eear
  \ee
for the forgetful functor and
  \be
  \begin{array}{rl}
  I \colon \quad \C &\!\!\! \xrightarrow{~~~} \M 
  \nxl1
  x &\!\!\! \longmapsto A \oti x
  \eear
  \ee
for the induction functor. In case a functor $F$ has a right or left adjoint, we denote
it by $F^\ra$ and $F^\la$, respectively.
 
As is well known, and easy to check, the mappings $ \vpsi\colon \Hom_\M(I(x),m) \To
\HomC(x,U(m))$ and $ \vphi\colon \HomC(x,U(m)) \To 
  \Hom_\M(I(x),m) $ defined by
  \be
  \vpsi(f) := f \circ (\eta \oti \id_x) \qquand
  \vphi(g) := \rho_m \circ (\id_A \oti g) \,.
  \label{eq:Ula=I}
  \ee
for $x\iN\C$ and $m \,{\equiv}\, (\dot m,\rho_m) \iN\M$, are each other's inverse, so that we have

\begin{lem}
$I$ is left adjoint to $U$, so that we can write $U^\la \eq I$.
\end{lem}

\begin{rem}
This statement is analogous to the classical result that for an embedding $R\,{\subset}\,S$ of
rings the induction functor $S \,{\otimes_R}, -$ is left adjoint to the forgetful functor
$U\colon S\Mod \To R\Mod$. In this setting, a right adjoint of $U$ is given by 
the coinduction functor $\Hom_R(S,-)$\,.
\end{rem}

\begin{lem}
If $A \eq (A,\mu,\eta,\Delta,\eps)$ is a Frobenius algebra, then
$I$ is right adjoint to $U$, so that we can write $U^\ra \eq I$.
\end{lem}

\begin{proof}
Consider the mappings $ \tpsi\colon \Hom\M(m,I(x)) \To \HomC(U(m),x) $ and
  $ \tphi\colon \HomC(U(m),x)
  $\linebreak[0]$
  \To \Hom_\M(m,I(x)) $  
that, for $x\iN\C$ and $m \,{\equiv}\, (\dot m,\rho_m) \iN\M$, are defined by
  \be
  \tpsi(f) := (\eps \oti \id_x) \circ f \qquand
  \tphi(g) := \big( \id_A \oti (g \cir\rho_m)\big) \circ
  \big( (\Delta\cir\eta) \oti \id_m \big) \,.
  \label{eq:Ura=I}
  \ee
(That $\tphi(g)$ is a morphism in \M\ follows by a twofold use of the Frobenius relation for
the product and coproduct together with unitality $A$.) These mappings are each other's inverse:
we have $\tpsi \cir \tphi \eq \id_{\HomC(U(m),x)}$ by first using the defining property
of the counit $\eps$ and then the compatibility of $\rho_m$ with the unit,
as well as $\tphi \cir \tpsi \eq \id_{\Hom_\M(m,I(x))}$ by first using the module morphism
property, then the Frobenius relation and then the defining properties of unit and counit.
\end{proof}

Thus for a Frobenius algebra the functors $I$ and $U$ are two-sided adjoints to each other;
in other words, induction and coinduction coincide. They thus form what is called
a strongly adjoint pair of functors \cite{mori2}, or a pair of \emph{Frobenius functors}, or
Frobenius pair \cite{cagN}.
Such functors are particularly well-behaved, e.g.\ they are exact, preserve limits 
and colimits, and preserve injective and projective objects (see e.g.\ \cite{caDm}).

%%%%%%%%%%%%%%%%%%%%%%%%%%%%%%%%%%%%%%%%%%%%%%%%%%%%%%%%%%%%%%%%%%%%%%%%

\subsection{The central monad}\label{app:centralmonad}

Consider the mapping
  \be
  c \,\longmapsto\, Z(c) := \int^{x\in\C}\!\! x^\vee \oti c \oti x \,,
  \label{eq:c-Z(c)}
  \ee
sending an object $c$ of \C\ to a coend as in formula \eqref{eq:Z(c)}. This 
furnishes an endofunctor of \C; moreover, this endofunctor carries a natural algebra 
structure and is thus a \emph{monad} on \C. 
  
\begin{Definition}[Monad]
A \emph{monad} $\T \eq (T,\mm,\eeta)$ on a category \C\ is an algebra in the monoidal category 
of endofunctors of \C, i.e.\ an endofunctor $T\colon \C\To\C$ together with a natural 
transformation $\mm \eq (\mm_c)_{c\in\C}\colon T\cir T \TO T$ and a natural transformation
$\eeta \eq (\eeta_c)_{c\in\C}\colon \id_\C \TO T$
that obey the associativity and unit properties
  \be
  \mm_c \circ T(\mm_c) = \mm_c \circ \mm_{T(c)} \qquand
  \mm_c \circ T(\eeta_c) = \id_{T(c)} = \mm_c \circ \eeta_{T(c)} 
  \label{eq:T-ass,id}
  \ee
for all $c \iN \C$.
\end{Definition}

Note that naturality of $m$ and $\eta$ mean that
  \be
  \mm_d \circ T^2(f) = T(f) \circ \mm_c \qquand
  \eeta_d \circ f = T(f) \circ \eeta_c
  \label{eq:m-eta-natural}
  \ee
for any morphism $f\colon c\To d$ in \C.
A virtue of the monad concept is that it does not require \C\ itself to be monoidal. In  case 
\C\  \emph{is} monoidal, then for any algebra $A\iN \C$ the endofunctors $- \oti A$ and 
$A \oti -$ admit natural structures of monads on \C.

Modules over monads are defined analogously as modules over algebras: A (left) module over
a monad $\T \eq (T,\mm,\eeta)$ on \C\ (also called a $\T$-algebra) consists of an object
$m$ of \C\ and a morphism $\rho_m\colon T(m) \To m$ such that $\rho_m \cir \mm_m \eq
\rho_m \cir T(\rho_M)\colon T^2(m) \To m$ and $\rho_m \cir \eeta_m \eq \id_m$.
A morphism $f\colon m\To n$ of \T-modules satisfies by definition
  \be
  f \circ \rho_m = \rho_n \circ T(f)
  \label{eq:Z-module-morph}
  \ee
with $\rho_m\colon T(m) \To m$ and $\rho_n\colon T(n) \To n$ the respective representation
morphisms.  

\medskip

For any finite tensor category \C\ the prescription \erf{eq:c-Z(c)} provides
a distinguished monad on \C, the \emph{central monad} $\ZZ$.
The product and unit of $\ZZ$ are directly defined in terms of the dinatural families 
$\iZ c$ of the coends $Z(c)$:
  \be
  \eeta_c := \iZ c_\one \qquand
  \mm_c \circ \iZZ c_y \circ (\id_{y^\vee_{}} \oti \iZ c_x \oti \id_y)
  := \iZ c_{x\otimes y} 
  \ee
(in the formula for $\mm$, use of the Fubini theorem for iterated coends is implicit).
It can be shown \Cite{Thm.\,8.13}{brVi5} 
that the Drinfeld center \ZC\ is equivalent to the
category $\ZZ\Mod$ of $\ZZ$-modules in \C\ as a braided monoidal category.
In particular, a morphism $f\colon m\To n$ of $\ZZ$-modules is a morphism in \ZC.

The central monad $Z$ is even a \emph{bimonad}; it has a comonoidal structure given by
morphisms $\eps^\ZZ\colon \Ze \To \one$ and 
$\Delta^{\!\ZZ}_{c\otimes c'}\colon Z(c\oti c') \To Z(c) \oti Z(c')$ for all $c,c' \iN \C$
that are defined by
  $  
  \eps^\ZZ \cir \iZ\one_x \,{:=}\, \ev_x
  $  
and 
  \be
  \Delta^{\!\ZZ}_{c\otimes c'} \circ \iZ{c\otimes c'}_x := ( \iZ c_x \oti \iZ{c'}_x )
  \circ (\id_{x^\vee_{}} \oti \id_c \oti \coev_x \oti \id_{c'} \oti \id_x) \,,
  \ee
respectively, for all $x\iN \C$.

\medskip

Also note that $\Ze \eq L$. Accordingly, structural insight about statements involving the 
object $L \iN \C$ can favorably be obtained by formulating them in terms of the more canonical
monad $\ZZ$. Here are a few examples: First, the algebra structure on $L$ is a specialization 
of the algebra structure on $\ZZ$, namely $\mulL \eq \mm_\one$ and $\etaL \eq \eeta_\one$.
Second, the convolution product \erf{tildalpha.tildbeta} on $\HomC(\one,L)$, which in the monad 
setting reads $\tild\alpha \cvp \tild\beta \eq \mm_\one \cir Z(\tild\beta) \cir \tild\alpha$,
is a special case of a product that exists on the morphism space $\HomC(\one,T(\one))$ for any
monad $\T$ on a monoidal category.
Third, the adjunction isomorphisms \erf{eq:1,L<->F,F-2} can be expressed as
  \be
  \vphi: \quad \HomC(\one,L) \ni\, \tild\alpha \longmapsto \mm_\one \circ Z(\tild\alpha)
  \,\in \HomZ(F,F)
  \label{eq:1,L->F,F}
  \ee
and
  \be
  \vpsi: \quad \HomZ(F,F) \ni\, g \longmapsto g \circ \eeta_\one \,\in \HomC(\one,L) \,,
  \label{eq:F,F->1,L}
  \ee
respectively, showing in particular that $\vphi(\tild\alpha)$ is indeed an element of
$\HomZ(F,F)$ rather than only of $\HomC(F,F) \,{\supset}\, \HomZ(F,F)$.
And fourth, that $L$ has a natural Hopf algebra structure if \C\ is braided can be seen,
without spelling out the structural morphisms explicitly, as a consequence of

\begin{lem}
For \C\ a braided finite tensor category, the endofunctors $\ZZ$ and $L\oti-$ are 
isomorphic as bimonads. A pair of mutually inverse isomorphisms
$\zeta\colon L\oti{-} \,{\Rightarrow}\,\ZZ$ and $\xi\colon \ZZ \,{\Rightarrow}\, L\oti-$
is given by
  \be
  \zeta_c \circ (\iL_x \oti \id_c) := \iZ c_x \circ (\id_{x^\vee_{}} \oti \cb_{x,c})
  \quad {\rm and} \quad
  \xi_c \circ \iZ c_x := (\iL_x \oti \id_c) \circ (\id_{x^\vee_{}} \oti \cb^{-1}_{x,c}) \,.
  \label{eq:xi_c}
  \ee
\end{lem}

\begin{proof}
That $\xi_c \cir \zeta_c \eq \id_{L\otimes c}$ and $\zeta_c \cir \xi_c \eq \id_{Z(c)}$
is seen by direct computation. That these natural transformations furnish an isomorphism of
bimonads follows from the fact that, via the braiding, the functors $x \oti - \oti y^\vee$ and
$x \oti y^\vee \oti -$ are isomorphic. Indeed, that the units and counits are intertwined
is immediate, and that the products and coproducts are intertwined is verified by the
calculations
  \be
  \bearll
  \multicolumn{2}{l} {
  \xi_c \circ \mm_c \circ \zeta_{Z(c)} \circ (\iZe_y \oti \zeta_c) \circ
  (\id_{y^\vee_{}} \oti \id_y \oti \iZe_x \oti \id_c) 
  }
  \Nxl2 \hspace*{1.2em} &
  = \xi_c \circ \mm_c \circ \iZZ c_y \circ (\id_{y^\vee_{}} \oti \cb_{y,Z(c)}) \circ
  \big( \id_{y^\vee_{}} \oti \id_y \oti [\iZ c_x \cir (\id_{x^\vee_{}} \oti \cb_{x,c})] \big)
  \Nxl2 &
  = \xi_c \circ \mm_c \circ \iZZ c_y \circ (\id_{y^\vee_{}} \oti \iZ c_x \oti \id_y)
  \circ (\id_{y^\vee_{}} \oti \id_{x^\vee_{}} \oti \cb_{x,c} \oti \id_y) 
  \Nxl2 & \hspace*{1.2em}
  \circ\, (\id_{y^\vee_{}} \oti \cb_{y,x^\vee\otimes x \otimes c})
  \Nxl2 &
  = \xi_c \circ \iZ c_{x\otimes y} \circ (\id_{y^\vee_{}} \oti \id_{x^\vee_{}}
  \oti \cb_{x\otimes y,c}) \circ (\id_{y^\vee_{}} \oti \cb_{y,x^\vee\otimes x} \oti \id_c)
  \Nxl2 &
  = (\iZe_{x\otimes y} \oti \id_c) \circ  (\id_{y^\vee_{}} \oti \id_{x^\vee_{}}
  \oti \cb_{x\otimes y,c}^{-1}) \circ (\id_{y^\vee_{}} \oti \id_{x^\vee_{}} \oti \cb_{x\otimes y,c})
  \Nxl2 & \hspace*{1.2em}
  \circ\, (\id_{y^\vee_{}} \oti \cb_{y,x^\vee\otimes x} \oti \id_c)
  \Nxl2 &
  = \big( \iZe_{x\otimes y} \cir (\id_{y^\vee_{}} \oti \cb_{y,x^\vee\otimes x}) \big) \oti \id_c
  \,=\, \big( \mul_L \cir (\iZe_y \oti \iZe_x) \big) \oti \id_c
  \eear
  \ee
and 
  \be
  \bearll
  \multicolumn{2}{l} {
  (\xi_c \oti \xi_{c'}) \circ \Delta^{\!Z}_{c\otimes c'} \circ \zeta_{c\otimes c'}
  \circ (\iZe_x \oti \id_c \oti \id_{c'})
  }
  \Nxl2 \hspace*{2.2em} &
  = (\iZe_x \oti \id_c \oti \iZe_x \oti \id_{c'}) 
  \circ\, (\id_{x^\vee_{}} \oti \cb^{-1}_{x,c} \oti id_{x^\vee_{}} \oti \cb^{-1}_{x,c})
  \Nxl2 & \hspace*{1.2em}
  \circ\, (\id_{x^\vee_{}} \oti \id_c \oti \coev_x \oti \id_{c'} \oti \id_x)
  \circ (\id_{x^\vee_{}} \oti \cb_{x,c\otimes c'})
  \Nxl2 &
  = (\iZe_x \oti \id_c \oti \iZe_x \oti \id_{c'})
  \circ (\id_{x^\vee_{}} \oti \id_x \oti \cb_{x^\vee_{},c} \oti \id_x \oti \id_{c'})
  \Nxl2 & \hspace*{1.2em}
  \circ\, (\id_{x^\vee_{}} \oti \cb_{x,c} \oti \id_{c'})
  \Nxl2 &
  = \big( (\id_L \oti \cb_{L,c}) \cir (\Delta_L \oti \id_c) \cir (\iZe_x \oti \id_c) \big)
  \oti \id_{c'}
  \eear
  \label{eq:coprod-intertw}
  \ee
respectively (the right hand side of \erf{eq:coprod-intertw}
defines the comonoidal structure of $L\oti-$).
\end{proof}

Actually we could choose to replace the braiding and inverse braiding
in \eqref{eq:xi_c} by any odd power of them. This would not have any effect on the
resulting structural morphisms for the Hopf algebra $L$, but only change the braiding
in the definition of the comonoidal structure of the bimonad $L\oti-$.

\newpage
%%%%%%%%%%%%%%%%%%%%%%%%%%%%%%%%%%%%%%%%%%%%%%%%%%%%%%%%%%%%%%%%%%%%%%%%

\newcommand\wb{\,\linebreak[0]} \def\wB {$\,$\wb}
\newcommand\Bi[2]    {\bibitem[#2]{#1}}
\newcommand\inBo[8]  {{\em #8}, in:\ {\em #1}, {#2}\ ({#3}, {#4} {#5}), p.\ {#6--#7} }
\newcommand\inBO[9]  {{\em #9}, in:\ {\em #1}, {#2}\ ({#3}, {#4} {#5}), p.\ {#6--#7} {\tt [#8]}}
\newcommand\J[7]     {{\em #7}, {#1} {#2} ({#3}) {#4--#5} {{\tt [#6]}}}
\newcommand\JO[6]    {{\em #6}, {#1} {#2} ({#3}) {#4--#5} }
\newcommand\JP[7]    {{\em #7}, {#1} ({#3}) {{\tt [#6]}}}
\newcommand\BOOK[4]  {{\em #1\/} ({#2}, {#3} {#4})}
\newcommand\PhD[2]   {{\em #2}, Ph.D.\ thesis #1}
\newcommand\Prep[2]  {{\em #2}, preprint {\tt #1}}
\def\adma  {Adv.\wb Math.}
\def\ajse  {Arabian Journal for Science and Engineering}
\def\amjm  {Amer.\wb J.\wb Math.}
\def\anma  {Ann.\wb Math.}
\def\apcs  {Applied\wB Cate\-go\-rical\wB Struc\-tures}
\def\aspm  {Adv.\wb Stu\-dies\wB in\wB Pure\wB Math.}
\def\atmp  {Adv.\wb Theor.\wb Math.\wb Phys.}   
\def\cocm  {Com\-mun.\wb Con\-temp.\wb Math.}
\def\coia  {Com\-mun.\wB in\wB Algebra}
\def\coma  {Con\-temp.\wb Math.}
\def\comp  {Com\-mun.\wb Math.\wb Phys.}
\def\cpma  {Com\-pos.\wb Math.}
\def\duke  {Duke\wB Math.\wb J.}
\def\ijmp  {Int.\wb J.\wb Mod.\wb Phys.\ A}
\def\imrn  {Int.\wb Math.\wb Res.\wb Notices}
\def\inma  {Invent.\wb math.}
\def\jajm  {Japan.\wb J.\wb Math.}
\def\jams  {J.\wb Amer.\wb Math.\wb Soc.}
\def\jgap  {J.\wb Geom.\wB and\wB Phys.}
\def\jhep  {J.\wb High\wB Energy\wB Phys.}
\def\joal  {J.\wB Al\-ge\-bra}
\def\jopa  {J.\wb Phys.\ A}
\def\jktr  {J.\wB Knot\wB Theory\wB and\wB its\wB Ramif.}
\def\jpaa  {J.\wB Pure\wB Appl.\wb Alg.}
\def\jram  {J.\wB rei\-ne\wB an\-gew.\wb Math.}
\def\leni  {Lenin\-grad\wB Math.\wb J.}
\def\maan  {Math.\wb Annal.}
\def\mams  {Memoirs\wB Amer.\wb Math.\wb Soc.}
\def\maze  {Math.\wb Zeit\-schr.}
\def\momj  {Mos\-cow\wB Math.\wb J.}
\def\nupb  {Nucl.\wb Phys.\ B}
\def\pajm  {Pa\-cific\wB J.\wb Math.}
\def\pams  {Proc.\wb Amer.\wb Math.\wb Soc.}
\def\phlb  {Phys.\wb Lett.\ B}
\def\nyjm  {New\wB York\wB J.\wb Math}
\def\plms  {Proc.\wB Lon\-don\wB Math.\wb Soc.}
\def\prja  {Proc.\wB Japan\wB Acad.}
\def\pspm  {Proc.\wb Symp.\wB Pure\wB Math.}
\def\quto  {Quantum\wB Topology}
\def\ruma  {Revista de la Uni\'on Matem\'atica Argentina}
\def\sema  {Selecta\wB Mathematica}
\def\slnm  {Sprin\-ger\wB Lecture\wB Notes\wB in\wB Mathematics}
\def\stkt  {Sci.\wb Rep,\wB Tokyo\wB Kyoiku\wB Daigaku\wB A}
\def\taac  {Theo\-ry\wB and\wB Appl.\wb Cat.}
\def\tams  {Trans.\wb Amer.\wb Math.\wb Soc.}
\def\thmp  {Theor.\wb Math.\wb Phys.}
\def\trgr  {Trans\-form.\wB Groups}

\end{document}